\documentclass[12pt]{amsart}
\usepackage{fullpage,tikz,scalerel}
\usepackage[utf8]{inputenc}
\usepackage{color,amsfonts,amsmath,amssymb,amsthm,cases,wasysym}
\usepackage[normalem]{ulem}
\usepackage{svg}
\usepackage{enumitem}
\usepackage{booktabs}
\usepackage{mathtools}
\usepackage{hyperref} 
\usepackage{float}

\title{Framing Triangulations for Arbitrary Integer Flow Polytopes}
\author{Jonah Berggren}
\email{jrberggren@uky.edu}
\date{}

\newtheorem{thm}{Theorem}[section]
\newtheorem{prop}[thm]{Proposition}
\newtheorem{lemma}[thm]{Lemma}
\newtheorem{cor}[thm]{Corollary}
\theoremstyle{definition}
\newtheorem{conj}[thm]{Conjecture}
\newtheorem{defn}[thm]{Definition}
\newtheorem{remk}[thm]{Remark}
\newtheorem{example}[thm]{Example}

\newtheorem{thmIntro}{Theorem}    
\newtheorem{propIntro}[thmIntro]{Proposition}

\newcommand{\inn}{\textup{in}}
\newcommand{\out}{\textup{out}}
\newcommand{\hinn}{\widehat{\textup{in}}}
\newcommand{\hout}{\widehat{\textup{out}}}

\newcommand{\F}{\mathcal F}	%flow polytope
\newcommand{\FF}{\mathfrak F}	%framing
\newcommand{\R}{\mathcal R}	%route-clique
\newcommand{\T}{\mathcal T}	%route-clique
\newcommand{\K}{\mathcal K}	%layering-simplex
\renewcommand{\L}{\mathcal L}	%layering-simplex
\newcommand{\M}{\mathcal M}	%layering-simplex
\renewcommand{\a}{\mathbf{a}}	%netflow vector
\newcommand{\ucong}{\overset{\textup{int}}{\cong}}
\newcommand{\I}{\mathcal{I}} %indicator vector of route
\newcommand{\J}{\mathcal{J}} %indicator vector of layering
\newcommand{\src}{\textup{src}} %post-source order
\newcommand{\routes}{\textup{routes}} %routes of layering-simplex
\newcommand{\D}{\Delta}		%polyhedral clique-simplex or layering-simplex
\newcommand{\rat}{\textup{rat}}	%rational
\renewcommand{\int}{\textup{int}}	%integer
\newcommand{\fs}{\textup{fs}}	%flow-supporting
\newcommand{\dkk}{\textup{DKK}}	%DKK
\newcommand{\con}{\textup{con}} %conservationist
\newcommand{\s}{s}	%source vertex
\newcommand{\G}{\Gamma}	%planar embedding

\begin{document}

\maketitle

\begin{abstract}
	Framing triangulations of unit flow polytopes have received a great deal of recent study with rich connections to various generalizations of Catalan and Cambrian combinatorics as well as volume and $h^*$-polynomial formulas.
	This story has largely been restricted to unit flow polytopes, with only two recent works giving descriptions of framing triangulations on classes of non-unit flow polytopes.
	In this article we introduce the first theory of (unimodular) framing triangulations for arbitrary integer flow polytopes.
	We will observe some pathologies in general examples which are impossible in the unit case, and propose in response a class of ``well-ordered'' framing triangulations which we expect to inherit key properties from the unit case while still containing all settings of framing triangulations existing in the literature.
\end{abstract}

\section{Introduction}

Flow polytopes, which model the space of \emph{flows} on a directed acyclic graph (DAG) $G$ with a fixed netflow vector $\a$ on its vertices, are a fundamental class of polytopes in combinatorial optimization~\cite{FRW,MM,RW,Zang} with connections to a plethora of areas of mathematics including 
Grothendieck polynomials~\cite{LMD}, algebraic geometry~\cite{EM,Hille}, and representation theory~\cite{BV}.

Many interesting polytopes arise as flow polytopes, such as Gelfand-Tsetlin polytopes \cite{LMD}, the Chan–Robbins–Yuen polytope~\cite{CRY}, Tesler polytopes~\cite{MMR}, and the Pitman–Stanley polytope~\cite{PS}.
\emph{Unit} flow polytopes, or flow polytopes defined from a \emph{unit} netflow vector $\a_1=(1,0,\dots,0,-1)$, have received a wealth of recent study due to the theory of framing triangulations.
Danilov, Karzanov, and Koshevoy~\cite{DKK} used the data of a \emph{framing} $\FF$ on a DAG $G$ to induce a regular unimodular \emph{framing triangulation} on the unit flow polytope $\F_G(\a_1)$.
The dual graph of a framing triangulation was given a lattice structure, called the \emph{framing lattice}, by von Bell and Ceballos~\cite{vBC}.
The combinatorics of framing triangulations and framing lattices recovers through special cases many classically studied areas of combinatorics.
Notably, many important families of lattices may be realized as framing lattices~\cite{vBC}, such as the Tamari lattice, the weak order on the symmetric group, and various generalizations of these lattices (including the type-A Cambrian lattices~\cite{ReadingCamb}, the $\nu$-Tamari lattices~\cite{PV,BDMY} and the principal order ideals in Young's lattice~\cite{BDMY}, the $s$-weak orders~\cite{DMPTY}, the permutree lattices~\cite{PP}, and the $\tau$-tilting posets of certain gentle algebras~\cite{WIWT,GENT,UQAM,PPP,BDMTY}).
Finally, the enumeration and structure of the maximal simplices of a framing triangulation open connections to the volume and $h^*$-polynomials~\cite{WIWT} of the underlying flow polytope.
{In particular, a concurrent article by Gonz\'alez D'Le\'on, Hanusa, and Yip~\cite{DHY} relates the combinatorics of framing triangulations to the Lidskii volume formulas of Baldoni and Vergne~\cite{BV}.}

The study of framing triangulations has thus far been limited to unit flow polytopes, with two exceptions.
The authors of the concurrent article~\cite{DHY} work in the generality of flow polytopes from a netflow vector $\a$ with a unique negative element, which generalizes the unit case. Among other things, they give a generalization of Danilov, Karzanov, and Koshevoy's framing triangulations to this setting and connect them with the Lidskii volume formulas of Baldoni and Vergne~\cite{BV} which hold for this family of flow polytopes.
In~\cite{PLAN}, the author developed a theory of framing triangulations given a \emph{strongly planar} embedding of a graph $G$ with a nontrivial netflow vector $\a$ (which may have multiple positive and multiple negative entries), generalizing the strongly planar framing triangulations studied by M\'esz\'aros, Morales, and Striker~\cite{MMS}.

The main contribution of this article is to develop a theory of framing triangulations for the flow polytope of a general DAG $G$ and arbitrary integer netflow vector $\a$ which simultaneously generalizes all of these existing theories.
{We will observe some pathologies that may arise in this general setting but not in the settings of~\cite{DKK,DHY,PLAN}; in response to this, we will propose a more specific class of \emph{well-ordered} framing triangulations which still serves as a common generalization of these settings.
We show that, unlike the general case, a well-ordered framing triangulation is always induced by a flag complex (i.e., a pairwise compatibility condition on integer points). We expect but do not prove that well-ordered framing triangulations preserve many important properties from the settings of~\cite{DKK,DHY,PLAN} -- e.g., a theory of framing posets, and connections to volume and $h^*$-polynomial formulas.}

\subsection{Results}

Let $G=(V,E)$ be a DAG with integer netflow vector $\a$.
We define an \emph{augmentation} $(\hat G,\hat\a)$ of $(G,\a)$ (Definition~\ref{defn:cover}) by adding \emph{inflow edges/vertices} and \emph{outflow edges/vertices} such that $\hat\a_i$ is 1 on the inflow vertices, negative on the outflow vertices, and 0 on the original vertices $V$, and such that contracting all inflow and outflow edges retrieves $(G,\a)$.
For example, the bottom-left of Figure~\ref{fig:intro} (netflow vector labelled in blue) is an augmentation of the top-left.
A \emph{framing} $\hat\FF$ on an augmentation $(\hat G,\hat\a)$ is a set of total orders of the incoming and outgoing edges to each internal vertex, along with a total order on the inflow vertices. We denote framings in examples using our embeddings of DAGs, with lower edges below higher edges in all framing orders and with inflow vertices ordered bottom to top.

\begin{figure}
	\centering
	\def\svgscale{.38}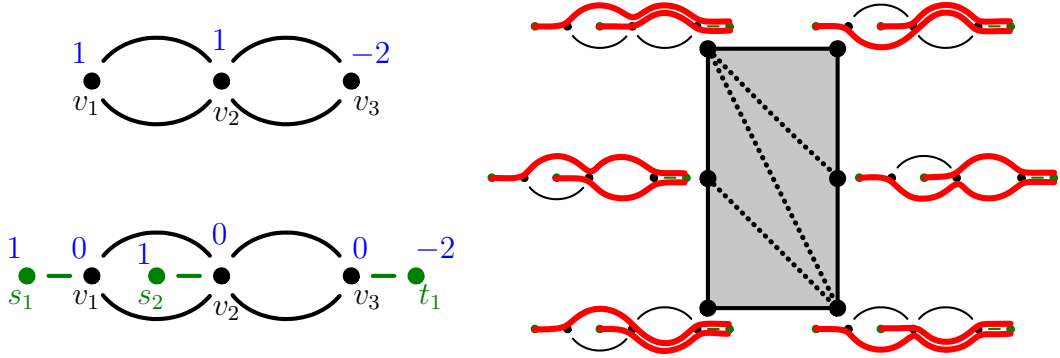
	\caption{A DAG $G$ with netflow vector $\a$ labelled in blue (top-left), a framed augmentation $(\hat G,\hat\a,\hat\FF)$ (bottom-left), and its framing-triangulated flow polytope (right).}
	\label{fig:intro}
\end{figure}

Analogous to the unit case of Danilov, Karzanov, and Koshevoy, a framing of an augmentation induces a notion of pairwise compatibility of routes (paths from source to sink). A \emph{route-clique} is a set of pairwise compatible routes. 
A \emph{layering} is a route-clique with one route starting at each inflow vertex $i$ and $|\hat\a_j|$ routes ending with each outflow vertex $j$. Taking the \emph{indicator vector} of a layering in $\mathbb R^E$ gives a bijection from layerings to integer points of the flow polytope $\F_G(\a)$.
In Figure~\ref{fig:intro}, the six integer points of the flow polytope on the right are labelled by the six layerings.

The framing $\hat\FF$ induces a total order $<_i^+$ on the routes beginning from any inflow vertex; applying lexicographical order then gives a total order $<_\src^+$ on layerings of $(\hat G,\hat\a,\hat\FF)$.
	A set $\L=\{L^1,\dots,L^m\}$ of layerings of $(\hat G,\hat\a,\hat\FF)$ indexed so that $L^1<_\src^+\dots<_\src^+L^m$ is a \emph{layering-simplex} of $(G,\a,\FF)$ if
		the set $\routes(\L)$ of routes used in $\L$ is a route-clique, and
		for any $i\in[m]$ we have 
			\[L^i=\min_{<_\src^+}\left\{L\in\text{Layerings}(\hat G,\hat\a,\hat\FF)\ :\ \routes(\{L\})\subseteq\routes(\{L^{i},\dots,L^m\})\right\}\textup{.}\]

In fact, the layering-simplices form an abstract simplicial complex on the set of layerings of $(\hat G,\hat\a,\hat\FF)$. 
{Note that unlike the framing triangulations of~\cite{DKK,DHY,PLAN} our notion of layering-simplex is not defined from a pairwise compatibility condition on the integer points of the flow polytope $\F_G(\a)$, and indeed it may fail to be a flag complex.}
We prove that letting $\check\D_1(\L)$ be the convex hull of the indicator vectors of the layerings of a layering-simplex $\L$ induces a unimodular triangulation of $\F_G(\a)$:

\begin{thmIntro}[{Theorem~\ref{thm:main-aug}}]
	\label{thm:i}
	Let $G$ be a DAG with netflow vector $\a$ and let $(\hat G,\hat\a,\hat\FF)$ be a framed augmentation of $(G,\a)$. The \emph{framing triangulation}
	\[
		\mathcal T(\hat G,\hat\a,\hat\FF):=\{\check\D_1(\L)\ :\ \L\in\textup{Layerings}(\hat G,\hat\a,\hat\FF)\}
	\]
	is a unimodular lattice triangulation of $\F_G(\a)$.
\end{thmIntro}

The proof of Theorem~\ref{thm:i} is ultimately an algorithm which, given $x>0$ and an $x\times\a$-flow $F$ of $G$, returns $F$ as a convex combination of indicator vectors of layerings within the same layering-simplex.

We will see an example (Example~\ref{ex:k33}) of a framing triangulation whose layering-simplex complex is not a flag complex (i.e., not induced from a pairwise compatibility condition on layerings) and whose dual graph is not the Hasse diagram of a poset -- two pathologies which may not arise in the unit case.
We say that a framed augmentation $(\hat G,\hat\a,\hat\FF)$ (or its framing triangulation) is \emph{well-ordered} if every layering gives the same map from inflow vertices to outflow vertices; note that the framed augmentation of Figure~\ref{fig:intro} is well-ordered because there is a unique outflow vertex.
{Moreover, the settings of framing triangulations discussed earlier in the introduction~\cite{DKK,DHY,PLAN} may all be described as well-ordered framing triangulations.}

\begin{propIntro}[{Proposition~\ref{prop:ltr}}]
	\label{prop:ltrI}
	Suppose $(\hat G,\hat\a,\hat\FF)$ is well-ordered and $\L$ be a set of layerings. Then $\L$ is a layering-simplex if and only if the layerings of $\L$ are \emph{pairwise noncrossing}: i.e., if for any choice of $L,M\in\L$, the set $\routes(\{L,M\})$ is a route-clique and without loss of generality $p<_i^+q$ for every $p\in L$ and $q\in M$ starting at the same inflow vertex $i$.
\end{propIntro}

{We conjecture that the dual graph of a well-ordered framing triangulation is the Hasse diagram of a poset, generalizing the framing lattices of von Bell and Ceballos~\cite{vBC} in the unit case. Moreover, we expect that some of the theory developed in~\cite{DHY} may generalize to the well-ordered setting. In particular, we hope that it may lead to a generalization of the Lidskii volume formulas of Baldoni and Vergne to the class of DAGs with netflow vectors admitting a well-ordered framed augmentation.}

The structure of this article is as follows. Section~\ref{sec:background} will give the relevant background on flow polytopes and on framing triangulations in the unit case as developed by Danilov, Karzanov, and Koshevoy~\cite{DKK}.
Section~\ref{sec:cons} contains the mathematical backbone of the article, proving a framing triangulation result for DAGs with \emph{conservationist} netflow vectors (of which augmentations are an example).
In Section~\ref{sec:aug} we will use this to get framing triangulations of an arbitrary integer flow polytope via framed augmentations.
In Section~\ref{sec:well-ordered} we will introduce and study well-ordered framing triangulations.
{Finally, in Section~\ref{sec:special-cases} we will show that the existing settings of framing triangulations~\cite{DKK,DHY,PLAN} all fall under the umbrella of our well-ordered framing triangulations.}

\subsubsection*{Acknowledgments}

The author would like to thank the anonymous referees of~\cite{PLAN}, who pointed out that a pathological flow polytope example (Example~\ref{ex:k33}) may still be triangulated if one drops the requirement for a pairwise compatibility condition, motivating the main result of this paper.
The author was supported by the NSF grant DMS-2451909. This work was supported by a grant from the Simons Foundation International [SFI-MPS-TSM-00013650, KS].

\section{Background}
\label{sec:background}

In this section we give some basic background on triangulations and integral equivalences of integer polytopes. Then we will define DAGs, netflow vectors, and the resulting flow polytopes. We will recall the framing triangulations of unit flow polytopes induced by framings on DAGs with one source and one sink developed by Danilov, Karzanov, and Koshevoy~\cite{DKK}. {We will leave reference of other settings of framing triangulations~\cite{DHY,PLAN} to Section~\ref{sec:special-cases}.}
This background section will largely follow~\cite{PLAN}.

\subsection{Integer polytope basics}

An \emph{integer polytope} is a polytope in $\mathbb R^n$ all of whose vertices are in $\mathbb Z^n$.
Following~\cite[\S2]{MMS} we say that two integer polytopes $P\subseteq\mathbb R^n$ and $Q\subseteq\mathbb R^m$ are \emph{integrally equivalent} if there is an affine transformation $\phi:\mathbb R^n\to\mathbb R^m$ whose restriction to $P$ is a bijection from $P$ to $Q$ that preserves the lattice. In other words, such that $\phi$ restricts to a bijection between $\mathbb Z^n\cap\text{aff}(P)$ and $\mathbb Z^m\cap\text{aff}(Q)$ where $\text{aff}(-)$ denotes affine span. We also say that the map $\phi$ is an \emph{integral equivalence} and that $P\ucong Q$. Integral equivalence is a notion of ``isomorphism'' on integer polytopes which is sometimes referred to as \emph{unimodular equivalence} in the literature.

We concern ourselves with triangulations of integer polytopes.

\begin{defn}\label{defn:triangulation}
	Let $P$ be a $d$-dimensional integer polytope. A \emph{(lattice) triangulation} of $P$ is a finite set $\mathcal T$ of integer simplices such that
	\begin{enumerate}
		\item $P=\cup_{\D\in\mathcal T}\D$,
		\item if $\D\in\mathcal T$ and $\D'$ is a face of $\D$, then $\D'\in\mathcal T$, and
		\item for any $\D_1,\D_2\in\mathcal T$, the intersection $\D_1\cap\D_2$ is a (possibly empty) face of both $\D_1$ and $\D_2$.
	\end{enumerate}
	The triangulation $\mathcal T$ is \emph{unimodular} if each simplex of $\mathcal T$ has normalized volume 1 within its affine span.
\end{defn}

\subsection{DAGs, netflow vectors, and flow polytopes}

A graph $G=(V,E)$ is a collection of vertices $V$ and a collection of directed edges $E$ between vertices of $V$. An edge $e\in E$ is considered to begin at its \emph{tail} $t(e)\in V$ and end at its \emph{head} $h(e)\in V$. The graph $G$ is a \emph{directed acyclic graph (DAG)} if it has no oriented cycles. 
When we draw directed graphs, we will by default orient all edges left from to right (e.g., in Figure~\ref{fig:cons} the edge $\beta_1$ has tail $t(\beta_1)=v_2$ and head $h(\beta_1)=v_3$).
If $v$ is a vertex of $G$, let $\inn(v)$ be the set of incoming edges to $v$ and let $\out(v)$ be the set of outgoing edges of $v$.
A \emph{source} of $G$ is a vertex with no incoming edges, and a \emph{sink} of $G$ is a vertex with no outgoing edges.
A \emph{route} of $G$ is a path from a source to a sink.
Let $G=(V,E)$ be a directed acyclic graph on the vertex set $V=[n]$ and let $\a:=(\a_i)_{i\in V}$ be a vector of integer weights on the vertex set, with vertex $i$ having weight $\a_i$.
We call $\a$ a \emph{netflow vector} of $G$.
A function $F:E\to\mathbb R_{\geq0}$ is a (nonnegative) \emph{$\a$-flow}, or merely a \emph{flow} if the netflow vector $\a$ is understood from the context, if for each $i\in[n]$ the equation
\[
	\sum_{(e:i\to j)\in E}F(e)-\sum_{(e:k\to i)\in E}F(e)=\a_i
\]
is satisfied.
The flow $F$ is \emph{integer} (resp. \emph{rational}) if all coordinates $F(e)$ for $e\in E$ are integers (resp. rational).

\begin{defn}
	The \emph{$\a$-flow polytope} $\F_G(\a)$ is the polytope consisting of nonnegative $\a$-flows of $G$.
\end{defn}

When $\a$ is an integer vector,  the $\a$-flow polytope $\F_G(\a)$ is a (possibly empty) integer polytope.

Intuitively, flow polytopes model flows through networks. For example, one may imagine each edge of $G$ to be a pipe, and the label $F(e)$ to be an amount of water (or any type of flow) traveling through that pipe in the direction of the edge.
The equation $\sum_{(e:i\to j)\in E}F(e)-\sum_{(e:k\to i)\in E}F(e)=\a_i$ at any vertex $i$ translates to the idea that flow is conserved moving through that vertex, and $\a_i$ flow is added to the system at that node (or $|\a_i|$ is subtracted if $\a_i<0$).

\begin{lemma}\label{lem:netflow_sum_zero}
	If there exists an $\a$-flow of $G$, then $\sum_{i\in[n]}\a_i=0$.
\end{lemma}
\begin{proof}
	Choose $F$ to be an $\a$-flow of $G$. Then
	\begin{align*}
		\sum_{i\in[n]}\a_i&=\sum_{i\in[n]}\left(\sum_{(e:i\to j)\in E}F(e)-\sum_{(e:k\to i)\in E}F(e)\right).
	\end{align*}
	In the sum on the right, for any edge $e:j\to k$ the term $F(e)$ appears once indexed by vertex $j$ and $-F(e)$ appears once indexed by vertex $k$. These terms then cancel out for each edge and the sum is 0.
\end{proof}

The \emph{strength} of a netflow vector $\a$ is $S_\a:=\sum_{i\in[m]\ :\ \a_i>0}\a_i$. Note that when there exists an $\a$-flow of $G$, Lemma~\ref{lem:netflow_sum_zero} implies that this is the same as $\sum_{i\in[m]\ :\ \a_i<0}|\a_i|$.

\subsection{Framing triangulations for DAGs with one source and one sink}
\label{ssec:backonesource}

In this subsection, we recall a unimodular triangulation of $\F_G(\a)$ given by Danilov, Karzanov, and Koshevoy~\cite{DKK} when $G$ is a DAG with one source and one sink and $\a=(1,0,\dots,-1)$.

For this subsection, $G$ will always be a DAG with one source and one sink. A \emph{nonnegative flow on $G$ of strength $x\geq0$} is a flow with netflow vector $(x,0,\dots,0,-x)$ (where $x$ labels the source, and $-x$ labels the sink).
A \emph{unit flow} is a nonnegative flow on $G$ of strength $1$.
We use the symbol $\hat0$ for the source vertex and $\hat1$ for the sink vertex; all other vertices are \emph{internal vertices}.
The \emph{unit flow polytope} $\F_G(1)$ is the polytope of unit flows. More generally, if $x\in\mathbb R_{\geq0}$ then $\F_G(x\times 1)$ is the polytope of flows of strength $x$, or the dilation of the unit flow polytope by $x$.
Vertices of $\F_G(1)$ are precisely the indicator vectors of {routes} of $(G,\a)$ (recall that a \emph{route} is a path from source to sink).

\begin{defn}\label{defn:DKK-framed-dag}
	Let $G=(V,E)$ be a DAG with one source and one sink. For each internal vertex $v$ of $G$, assign a linear order to the edges in $\inn(v)$ and assign a linear order to the edges in $\out(v)$. This assignment is called a \emph{DKK-framing} of $G$, which we denote by $\FF$. We call a DAG $G$ with a DKK-framing $\FF$ a \emph{DKK-framed DAG}. If $e$ is less than $f$ in the linear order for $\FF$ on $\inn(v)$, we write $e\prec_{\FF,\inn(v)}f$ (and similarly for $\out(v)$). When $\FF$ and/or $\inn(v)$ or $\out(v)$ is clear, we may drop one or both subscripts.
\end{defn}

We notate a DKK-framing $\FF$ by labelling every internal half-edge of the DAG $G$ with a number. See Figure~\ref{fig:square} for an example (where all edges are oriented left to right): we have $\beta_1\prec_{\FF,\out(v)}\beta_2$ because the tail-label of $\beta_1$ is lower than that of $\beta_2$.

A DKK-framing on a DAG induces a notion of pairwise compatibility on its routes:

\begin{defn}\label{defn:compat0}
	Let $v$ be an internal vertex of a DKK-framed DAG $(G,\FF)$. We define the \emph{post-$v$-order} $\prec_v^+$ on the set of paths from $v$ to the sink $\hat1$ of $G$.
	Let $p$ and $q$ be distinct paths from $v$ to the sink.
	Let $\sigma=e_1\dots e_m$ be the maximal common subpath of $p$ and $q$ beginning at $v$. Suppose without loss of generality that $p$ contains $\sigma f$ and $q$ contains $\sigma g$, where $f$ is less than $g$ in $\prec_{\FF,\text{out}(h(e_m))}$.
	In this case, we say that $p\prec_v^+q$.
	This defines a total order on paths from $v$ to a sink.

	Dually, we define the \emph{pre-$v$-order} $\prec_v^-$ on the set of paths from a source to $v$.
	If $p$ and $q$ are distinct paths from a source to $v$, then let $\sigma=e_1\dots e_m$ be the maximal common subpath of $p$ and $q$ ending at $v$.
	If $p$ contains $f\sigma$ and $q$ contains $g\sigma$ where $f\prec_{\FF,\inn(t(e_1))}g$, then $p\prec_v^-q$.
\end{defn}

If $p$ and $q$ are routes of $(G,\FF)$ which both contain an internal vertex $v$, then we say that $p\prec_v^+q$ if $p_v^+\prec_v^+q_v^+$, where $p_v^+$ (resp. $q_v^+$) is the subpath of $p$ (resp. $q$) from $v$ to the sink. If $p$ and $q$ agree after the vertex $v$, then $p=_v^-q$ (even if $p$ and $q$ differ before the vertex $v$). We may similarly write $p\prec_v^-q$ or $p=_v^-q$.
In this way, $\prec_v^+$ is a total preorder on routes of $(G,\FF)$ passing through $v$ (i.e., a binary relation which is reflexive, transitive, and total, but not necessarily antisymmetric).

\begin{defn}\label{defn:compat0.5}
	Let $p$ and $q$ be routes of $(G,\FF)$ which both contain a vertex $v$. The routes $p$ and $q$ are \emph{incompatible} at $v$ if, without loss of generality, $p\prec_v^-q$ and $q\prec_v^+p$. The routes $p$ and $q$ are \emph{incompatible} if they are incompatible at any shared vertex, otherwise they are \emph{compatible}.
	A \emph{clique} is a set of pairwise compatible routes of $G$.
\end{defn}

Note that if $p$ and $q$ agree to the left of a vertex $v$, then $p=_v^-q$ and they cannot be incompatible at $v$.

The following definition and theorem connect cliques on $G$ to the unit flow polytope $\F_G(1)$.

\begin{defn}
	If $\R$ is a clique of $(G,\FF)$, then a \emph{($\R$-)clique combination} of $G$ is a linear combination
	\[\sum_{p\in\R}a_p\I(p),\]
	where $a_p\geq0$ for all $p\in\R$. It is \emph{positive} if $a_p>0$ for all $p\in\R$, and \emph{unit} if $\sum_{p\in\R}a_p=1$.
	The set of unit flows arising as (necessarily unit) $\R$-clique combinations is the \emph{polyhedral clique simplex} $\D_1(\R)$.
\end{defn}

We now phrase~\cite[Theorem 1]{DKK}. Note that the terminology of clique combinations is our own, and that a ``nonnegative flow'' in the sense of~\cite{DKK} is any nonnegative flow with netflow vector $(x,0,\dots,0,-x)$.

\begin{thm}[{\cite[Theorem 1]{DKK}}]\label{thm:DKK}
	Let $F$ be a strength-$x\geq0$ flow of a DKK-framed DAG $(G,\FF)$ (i.e., a flow with netflow vector $(x,0,\dots,0,-x)$). Then there is a unique positive clique combination $F=\sum_{p\in\R}a_p\I(p)$ for $F$. Moreover, if $F$ is integer-valued, then all coefficients $a_p$ are integers.
\end{thm}

Theorem~\ref{thm:DKK} implies that the polyhedral clique simplices of maximal cliques of $(G,\FF)$ form a unimodular triangulation of $\F_G(1)$. More specifically, if $\R$ is a clique of $(G,\FF)$ then define $\D_1(\R):=\text{conv}\{\I(p)\ :\ p\in\R\}$.
\begin{cor}\label{cor:DKK}
	Let $(G,\FF)$ be a DKK-framed DAG. Then $\D_1(\R)$ is a unimodular simplex with vertices $\{\I(p)\ :\ p\in\R\}$ for every clique $\R$ of $(G,\FF)$, and the set
	\[\mathcal T=\{\D_1(\R)\ :\ \R\text{ is a clique of }(G,\FF)\}\]
	is a unimodular triangulation of $\F_G(1)$.
\end{cor}

We call this triangulation $\T$ the \emph{(DKK-)framing triangulation} of $\F_G(1)$. It is also referred to as the \emph{DKK triangulation} in the literature.
See Example~\ref{ex:square}.
Note that it is also shown in~\cite{DKK} that these framing triangulations are regular, though we do not extend regularity to general integer flow polytopes in the present paper.

\begin{example}\label{ex:square}
	See the DKK-framed DAG of Figure~\ref{fig:square}. Edges are labelled in black and the DKK-framing at the unique internal vertex $v$ is labelled in red. Its flow polytope is integrally equivalent to a square, shown on the right of the figure.
	The routes $\alpha_1\beta_2$ and $\alpha_2\beta_1$ are incompatible at the vertex $v$, as $\beta_1\prec_{v}^+\beta_2$ but $\alpha_1\prec_{v}^-\alpha_2$. Any other choice of two routes is compatible. It follows that the two maximal cliques of $G$ are $\{\alpha_1\beta_1,\alpha_1\beta_2,\alpha_2\beta_2\}$ and $\{\alpha_1\beta_1,\alpha_2\beta_1,\alpha_2\beta_2\}$, shown on the middle of the figure. Each maximal clique $\R$ corresponds to a triangle (two-dimensional polyhedral simplex) of the flow polytope whose vertices are the indicator vectors of the routes of $\R$; see the framing triangulation of the flow polytope on the right of the figure.
	\begin{figure}
		\centering
		\def\svgscale{.38}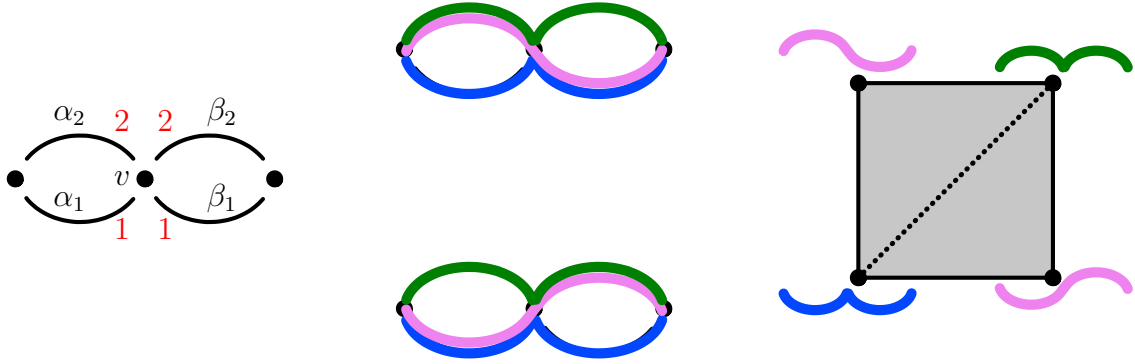
		\caption{A DKK-framed DAG, its two maximal cliques, and its framing-triangulated flow polytope.}
		\label{fig:square}
	\end{figure}
\end{example}

There have been two recent works extending DKK-framing triangulation results to flow polytopes from DAGs outside of the unit one-source-one-sink case. {In a major concurrent paper, 
Gonz\'alez D'Le\'on, Hanusa, and Yip~\cite{DHY} (among other things) give a generalization of DKK-framing triangulations to the setting of flow polytopes whose netflow vectors have only one negative entry. Similarly, the author~\cite{PLAN} obtained a theory of framing triangulations for DAGs with multiple sources, multiple sinks, and arbitrary netflow vectors under certain planarity conditions. These results are not necessary for the proofs of this paper, so we do not give them as background here. In Section~\ref{sec:special-cases} we will show that our framing triangulations are the same as those appearing in~\cite{DHY,PLAN} in the relevant settings.}

\section{Framing Triangulations in Conservationist Generality}
\label{sec:cons}

We will develop a theory of framing triangulations for arbitrary integer flow polytopes. For convenience, we will begin by restricting ourselves to the \emph{conservationist} case where flow is conserved at every internal vertex; up to integral equivalence this conservationist setting achieves all integer flow polytopes.
In this setting we will define framings directly on a DAG $G$ with conservationist netflow vector $\a$ without need for the the extra data of an augmentation of $G$ as referenced in the introduction. We will apply the results of Danilov, Karzanov, and Koshevoy~\cite{DKK} on the unit one-source-one-sink case to decompose an arbitrary $\a$-flow as a convex combination of indicator vectors of pairwise compatible routes, before grouping the routes together into \emph{layerings} to achieve framing triangulations.

\subsection{Framings on DAGs with conservationist netflows}

We first reduce the general case to the conservationist case and define a notion of framing in this setting.

\begin{defn}
	A netflow vector $\a$ of $G$ is \emph{conservationist} if
	\begin{enumerate}
		\item every source vertex $i$ satisfies $\a_i=1$,
		\item every sink vertex $i$ satisfies $\a_i<0$, and
		\item every internal vertex $i$ satisfies $\a_i=0$.
	\end{enumerate}
	Note that if $\a$ is conservationist then the strength $S_\a$ is the number of source vertices of $G$.
\end{defn}

We will see in the future (Corollary~\ref{cor:everything-is-conservationist}) that any integer flow polytope may be realized up to integral equivalence through a conservationist netflow vector.
{One could prove this at this time by first restricting to the subgraph induced by edges which may support nonzero flow (in particular, so that every source has positive netflow and every sink has negative netflow)} before performing some decontraction steps to satisfy the netflow conditions (or passing to an augmentation), but we omit this here for the sake of space.

We now set up some basic notation for a DAG $G$ with a conservationist netflow $\a$. We will always assume that $G$ has $m$ source vertices with labels $\{1,\dots,m\}$. Moreover, let $S_\a:=\sum_{i=1}^m\a_i$ be the strength of $\a$.

\begin{defn}\label{defn:framing}
	Let $G$ be a DAG with conservationist netflow vector $\a$. A \emph{framing} $\FF$ of $(G,\a)$ is the data of
	\begin{enumerate}
		\item for each non-sink vertex $i$, a total order $<_{\FF,\out(i)}$ on the outgoing edges $\out(i)$,
		\item for each internal vertex $i$, a total order $<_{\FF,\inn(i)}$ on the incoming edges $\inn(i)$, and
		\item a total order $<_{\FF,\src}$ on the source vertices of $G$.
	\end{enumerate}
	When $\a$ is conservationist, we will typically label the source vertices of $G$ as $\s_1,\s_2,\dots,\s_{S_\a}$ such that $\s_1<_{\FF,\src}\dots<_{\FF,\src}\s_{S_\a}$.
\end{defn}

In the future, we will say that a tuple $(G,\a,\FF)$ is \emph{conservationist} if $G$ is a DAG, $\a$ is a conservationist netflow vector, and $\FF$ is a framing of $(G,\a)$.

\begin{remk}\label{remk:embedding}
	When we draw examples of framed conservationist tuples $(G,\a,\FF)$, we will always use the embedding of $G$ to rank lower edges below higher edges in the orders $<_{\FF,\inn(i)}$ and $<_{\FF,\out(i)}$ and we will use the labels on source vertices to mark the order $<_{\FF,\src}$. For example, the left of Figure~\ref{fig:cons} shows a dag $G$ with conservationist netflow vector $\a=(1,1,0,-2)$. We will equip $(G,\a)$ with the framing $\FF$ which ranks, for example, the lower edge $\alpha_1$ below the higher edge $\beta_1$ in $<_{\FF,\inn(v_3)}$ and similarly ranks $\alpha_2<_{\FF,\out(v_3)}\beta_2$. The labels $v_1$ and $v_2$ of source vertices indicate that we will rank $v_1<_{\FF,\src}v_2$.
\end{remk}

\begin{remk}\label{remk:DKK}
	Note the similarity between our framings and the DKK-framings of Definition~\ref{defn:DKK-framed-dag}. If $G$ is a DAG with one source and one sink and $\a=(1,0,\dots,0,-1)$ is the unit netflow vector, then the total order $<_{\FF,\src}$ carries no information. Note, however, that in this case our framings will contain an order $<_{\FF,\out(1)}$ at the outgoing edge of the source vertex, while DKK-framings do not. In the case of unit flow polytopes, the information of $<_{\FF,\out(1)}$ will not affect the framing triangulation induced by $\FF$, but we will need outgoing framing orders to source vertices when there are multiple sources.
\end{remk}

\subsection{Route-clique combinations via the one-source-one-sink case}

In this subsection, we will use a framing $\FF$ on a DAG $G$ with conservationist netflow vector $\a$ to uniquely decompose any $x\times\a$-flow as a certain nonnegative combination of indicator vectors of routes.
First, we will generalize the post-$v$ and pre-$v$ orders of the one-source-one-sink case given in Definition~\ref{defn:compat0}.

\begin{defn}\label{defn:compat1}
	Let $(G,\a,\FF)$ be conservationist.
	Let $i$ be a vertex of $G$.
	We define the \emph{post-$i$-order} $<_i^+$ on the set of paths from $i$ to a sink of $G$.
	Let $p$ and $q$ be distinct paths from $i$ to a sink.
	Let $\sigma=e_1\dots e_m$ be the maximal common subpath of $p$ and $q$ beginning with $i$. Say $p$ contains $\sigma f$ and $q$ contains $\sigma g$, where $f<_{\R,\out(h(\sigma))}g$.
	In this case, we say that $p<_i^+q$.
	This defines a total order on paths from $i$ to a sink.

	Dually, if $i$ is not a sink then we define the \emph{pre-$i$-order} $<_i^-$ on the set of paths from a source to $i$.
	If $p$ and $q$ are distinct paths from a source to $i$, then let $\sigma=e_1\dots e_m$ be the maximal common subpath of $p$ and $q$.
	If $p$ contains $f\sigma$ and $q$ contains $g\sigma$, where $f<_{\R,\inn(t(\sigma))}g$, then $p<_i^-q$.

	If $p$ and $q$ are routes of $G$ which both contain the internal vertex $i$, then we say that $p<_i^+q$ if $p_i^+<_i^+q_i^+$, where $p_i^+$ (resp. $q_i^+$) is the subpath of $p$ (resp. $q$) from $i$ to a sink. If $p$ and $q$ agree after the vertex $i$, then $p=_i^+q$ (even if $p$ and $q$ differ before the vertex $i$). We may similarly write $p<_i^-q$ or $p=_i^-q$.
\end{defn}

Because the framing $\FF$ does not give an order to the incoming edges of a sink vertex, there is no pre-$i$ order if $i$ is a sink. Note that the framing does give an order to the outgoing edges of a source vertex, so there is a post-$i$ order at every source vertex.

The \emph{indicator vector} $\I(p)$ of a route $p$ is the vector in $\mathbb R^E$ with 1's at the edges used by $p$ and 0's elsewhere. note that when $\a=(1,0,\dots,0,-1)$ is the unit netflow vector, the indicator vector of any route is an $\a$-flow, but for other netflow vectors this is not the case.

\begin{defn}\label{defn:compat2}
	Let $p$ and $q$ be routes of $G$ and let $i$ be a vertex contained in both $p$ and $q$. The routes $p$ and $q$ are \emph{incompatible} at $i$ if, without loss of generality, $p<_i^-q$ and $q<_i^+p$. The routes $p$ and $q$ are \emph{incompatible} if they are incompatible at any shared vertex, otherwise they are \emph{compatible}.
	A \emph{route-clique} is a set of pairwise compatible routes of $G$.
	A \emph{route-clique combination} is a nonnegative linear combination
	\[\sum_{p\in \R}a_p\I(p),\]
	where $\R$ is a route-clique and $a_p\geq0$ for all $p\in\R$.
	The route-clique combination is \emph{positive} if $a_p>0$ for all $p\in\R$.
\end{defn}

We will now uniquely present an arbitrary $\a$-flow of $G$ as a route-clique combination (i.e., decompose it uniquely as a positive combination of indicator vectors of pairwise compatible routes). To do this, we will appeal to the framing triangulation result of Danilov, Karzanov, and Koshevoy~\cite[Theorem 1]{DKK}.

\begin{defn}
	Let $(G,\a,\FF)$ be conservationist.
	Let $\overline G$ be the DAG obtained by identifying all source vertices of $G$ and identifying all sink vertices of $G$. The framing $\FF$ descends to a DKK-framing $\overline\FF$ on $\overline G$ by forgetting the outgoing orders at the sources of $G$ and the total order $<_{\FF,\src}$.
	In this way, $(G,\a,\FF)$ gives rise to a DKK-framed DAG $(\overline G,\overline\FF)$ which we call the \emph{two-point identification} of $(G,\a,\FF)$.
	See Figure~\ref{fig:two_point_extension}, where the framing orders $<_{\FF}$ and $<_{\overline\FF}$ are induced from the embeddings (see Remark~\ref{remk:embedding}).
\end{defn}

\begin{figure}
	\centering
	\def\svgscale{.35}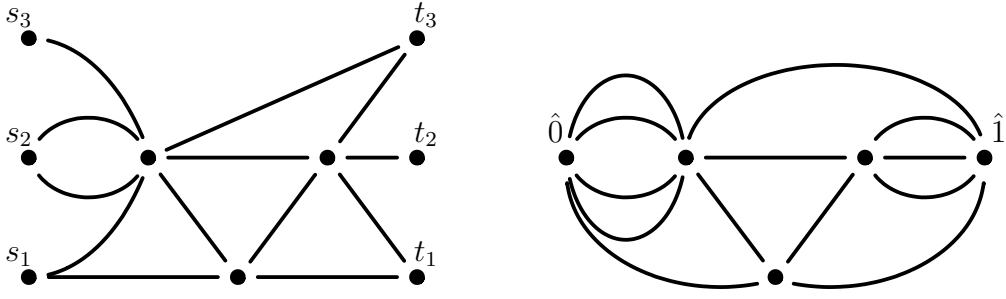
	\caption{A conservationist $(G,\a,\FF)$ (left) and its two-point identification (right).}
	\label{fig:two_point_extension}
\end{figure}

Note that a route of $G$ is precisely the same as a route of $\overline G$. Moreover, this respects compatibility:
\begin{remk}\label{remk:clique-dkk-bij}
	A set $\R$ of routes of $G$ is a route-clique of $(G,\a,\FF)$ (in the sense of Definition~\ref{defn:compat1}) if and only if $\R$ is a clique of $(\overline G,\overline\FF)$ (in the sense of Definition~\ref{defn:compat0.5}).
\end{remk}

Let $F$ be an $x\times\a$-flow of $G$ for some $x\geq0$.
Through the natural identification of edges of $G$ with edges of $\overline G$, we may consider $F$ to be a labelling of the edges of $\overline G$; it is then immediate that this is a $x\times S_\a$-flow of $\overline G$.
Note on the other hand that a strength-$(x\cdot S_\a)$ flow of $\overline G$ may not be an $x\times\a$-flow of $G$.
This shows an inclusion $\F_G(\a)\subseteq\F_{\overline G}(S_\a\times1)$
(the latter polytope being the $S_\a$-dilation of $\F_{\overline G}(1)$).
More specifically, the flow polytope $\F_G(\a)$ is the intersection of $\F_{\overline G}(S_\a\times1)$  with the hyperplanes $H_i=\{F\in\mathbb R^E\ :\ \a_i=\sum_{e\in\out(i)}F(e)\}$ where $i$ ranges over the source vertices of $G$.

\begin{prop}\label{prop:route-clique-combinations}
	Let $(G,\a,\FF)$ be conservationist. Let $F$ be an $x\times\a$-flow of $G$.
	Then there is a unique positive route-clique combination $F=\sum_{p\in \R}a_p\I(p)$ for $F$. Moreover, if $F$ is integer-valued, then all coefficients $a_p$ are integers.
\end{prop}
\begin{proof}
	Theorem~\ref{thm:DKK} gives the desired unique decomposition result for $\F_{\overline G}(1)$. Restricting this result to the $\a$-flows of $G$ and applying Remark~\ref{remk:clique-dkk-bij} translates this to the desired result for $\F_{G}(\a)$.
\end{proof}

{Grouping flows of $(G,\a,\FF)$ by which routes are used in their positive route-clique combination subdivides $\F_G(\a)$ as the restriction of the framing triangulation of $\F_{\overline G}(1)$ to $\F_G(\a)$. This subdivision of $\F_G(\a)$ is not in general a triangulation; we will now use the notion of \emph{layering} to refine it into a triangulation.}

\subsection{Layerings and framing triangulations}

\begin{defn}\label{defn:layering}
	Let $(G,\a,\FF)$ be conservationist. A \emph{layering} $L$ of $(G,\a,\FF)$ is a set of routes of $G$ such that
	\begin{enumerate}
		\item for every source vertex $i$ of $G$, there is precisely $\a_i=1$ route of $L$ beginning at $i$, and
		\item for every sink vertex $j$ of $G$, there are precisely $|\a_j|$ routes of $L$ ending at $j$.
	\end{enumerate}
	Recall that we index the source vertices of $G$ as $\s_1<_{\FF,\src}\dots<_{\FF,\src}\s_{S_\a}$.
	We will index a layering $L$ as $L=\{L_1,\dots,L_{S_\a}\}$, where for $i\in[S_\a]$ the route $L_i$ begins at $\s_i$.
	The set of layerings of $(G,\a,\FF)$ is notated $\textup{Layerings}(G,\a,\FF)$.
\end{defn}

If $L$ is a layering, define its \emph{indicator vector} $\J(L):=\sum_{p\in L}\I(p)$.

\begin{lemma}\label{lem:layering-int-flow}
	The map $\J:L\mapsto\J(L)$ gives a bijection from layerings of $(G,\a,\FF)$ to integer $\a$-flows of $G$.
\end{lemma}
\begin{proof}
	It is immediate from the definition of a layering that the indicator vector $\J(L)$ of a layering $L$ is an integer $\a$-flow.
	On the other hand, let $F$ be an integer $\a$-flow. Proposition~\ref{prop:route-clique-combinations} retrieves $F$ as a route-clique combination
	$
		F=\sum_{p\in\R_F}a_p\I(p)
	$
	with integer coefficients $a_p$.
	Because $F$ is an $\a$-flow, we have $\sum_{p\in\R_F\ :\ t(p)=i}a_p\I(p)=\a_i=1$ for any source vertex $i$ of $G$, hence all coefficients $a_p$ are equal to 1 and there is exactly one route $p_i$ of $\R_F$ starting any given source vertex.
	Similarly, for a sink vertex $j$ of $G$ we have
$\sum_{p\in\R_F\ :\ t(p)=i}\I(p)=-\a_j=|\a_j|$, hence there are $|\a_j|$ routes of $\R_F$ ending at $j$. It follows that $\R_F$ is a layering.
	Moreover, uniqueness of positive route-clique combinations as given in Proposition~\ref{prop:route-clique-combinations}
	ensures that $L\mapsto\J(L)$ and $F\mapsto\R_F$ are mutual inverses, and hence bijections.
\end{proof}

In light of Lemma~\ref{lem:layering-int-flow}, we may use layerings as our combinatorial stand-in for integer $\a$-flows, similar to the function of routes in the unit one-source-one-sink case. We will now work towards defining a notion of a layering-simplex, generalizing cliques of routes in the DKK-framing setting.

\begin{defn}
	Let $(G,\a,\FF)$ be conservationist. Let $L$ and $M$ be distinct layerings of $(G,\a,\FF)$. Choose $i\in[S_\a]$ maximal such that $L_i\neq M_i$ and suppose without loss of generality that $L_i<_{\s_i}^+M_i$. We say that $L<_\src^+M$. It is immediate that $<_\src^+$ is a total order on layerings, which we call the \emph{post-source order} on layerings.
\end{defn}

For example, Figure~\ref{fig:k33order} shows the order $<_\src^+$ of the six layerings of the conservationist $(K_{3,3},\a,\FF)$ depicted in Figure~\ref{fig:k33} (Example~\ref{ex:k33}).

If $\L=\{L^1,\dots,L^m\}$ is a set of layerings, let
$\routes(\L):=\{L^i_j\ :\ i\in[m],\ j\in[S_\a]\}=\cup_{i\in[m]}L^i$ be the underlying set of routes of $\L$.

\begin{defn}\label{defn:layering-simplex}
	Let $(G,\a,\FF)$ be conservationist. Let $\L=\{L^1,\dots,L^m\}$ be a set of layerings of $(G,\a,\FF)$, ordered so that $L^1<_\src^+\dots<_\src^+L^m$. We say that $\L$ is a \emph{layering-simplex} of $(G,\a,\FF)$ if
	\begin{enumerate}
		\item the set $\routes(\L)$ is a route-clique, and
		\item for any $i\in[m]$, we have 
			\[L^i=\min_{<_\src^+}\left\{L\in\text{Layerings}(G,\a,\FF)\ :\ \routes(\{L\})\subseteq\routes(\{L^{i},\dots,L^m\})\right\}.\]
	\end{enumerate}
\end{defn}

It is immediate that the set of layering-simplices of $(G,\a,\FF)$ is a simplicial complex.

Note that the compatibility condition on routes in the DKK-framed setting (Definition~\ref{defn:compat0.5}) is defined pairwise on routes, unlike Definition~\ref{defn:layering-simplex}. We will see in Example~\ref{ex:k33} that in general the simplicial complex of layering-simplices may fail to be a flag complex (i.e., it may fail to be induced by a pairwise compatibility condition on layerings).
In Section~\ref{sec:well-ordered}, we will provide conditions on a framed augmentation to be ``well-ordered,'' which will guarantee that the layering-simplex complex is a flag complex.

\begin{remk}
	It is immediate that a set $\L$ of layerings of $(G,\a,\FF)$ is a layering-simplex if and only if for any subset $\L'\subseteq\L$, we have
	\[
		\min_{<_\src^+}(\L')
		=
		\min_{<_\src^+}(
	\left\{L\in\text{Layerings}(G,\a,\FF)\ :\ \routes(\{L\})\subseteq\routes(\L')\right\}.
	\]
\end{remk}

\begin{lemma}\label{lem:route-deletion}
	Let $\L=\{L^1,\dots,L^m\}$ be a layering-simplex of conservationist $(G,\a,\FF)$. Then for any $i\in[m]$, there exists at least one route of $L^i$ which is not in $\routes(\{L^{i+1},\dots,L^{m}\})$.\footnote{The original arXiv submission of this work included the statement of Lemma~\ref{lem:route-deletion} as a (superfluous) third condition of the definition of a layering-simplex; this version has been updated to show that Lemma~\ref{lem:route-deletion} is a consequence of the two conditions of Definition~\ref{defn:layering-simplex}.}
\end{lemma}
\begin{proof}
	Suppose to the contrary that there exists $i\in[m]$ such that all routes of $L^i$ are in $\routes(\{L^{i+1},\dots,L^m\})$. Since the latter set is nonempty, we must have $i<m$. Then 
	\begin{align*}
		L^i&=\min_{<_\src^+}\left\{L\in\text{Layerings}(G,\a,\FF)\ :\ \routes(\{L\})\subseteq\routes(\{L^{i},\dots,L^m\})\right\} \\
		&=\min_{<_\src^+}\left\{L\in\text{Layerings}(G,\a,\FF)\ :\ \routes(\{L\})\subseteq\routes(\{L^{i+1},\dots,L^m\})\right\} \\
		&=L^{i+1},
	\end{align*}
	where the first equality comes from Definition~\ref{defn:layering-simplex}~(2) applied at $i$, and the last comes from Definition~\ref{defn:layering-simplex}~(2) applied at $i+1$.
	This is a contradiction, ending the proof.
\end{proof}

\begin{defn}
	An \emph{($\L$)-layering-simplex combination} is a nonnegative linear combination
	\[\sum_{L\in \L}a_L\J(L),\]
	where $\L$ is a layering-simplex and $a_L\geq0$ for all $L\in\L$.
	The layering-simplex combination is \emph{positive} if $a_L>0$ for all $L\in\L$.
	If $\L$ is a layering, then the \emph{polyhedral layering-simplex} $\D_1(\L)$ is the set of all flows realized as an $\L$-layering-simplex combination. We will wait until Theorem~\ref{thm:triangulation} to show that it is, in fact, a polyhedral simplex.
\end{defn}

We will see in the future that the set of polyhedral layering-simplices of a conservationist $(G,\a,\FF)$ are a unimodular triangulation of the flow polytope $\F_G(\a)$. We give an example of this now in the interest of clarifying Definition~\ref{defn:layering-simplex}.

\begin{example}\label{ex:consexample}
	Let $G$ be the DAG on the left of Figure~\ref{fig:cons} equipped with the conservationist netflow vector $\a=(1,1,0,-2)$. Let $\FF$ be the framing with $v_1<_{\FF,\src}v_2$ and whose orders $<_{\FF,\inn(v)}$ and $<_{\FF,\out(v)}$ are induced from the planar embedding (Remark~\ref{remk:embedding}) with lower edges appearing earlier in these orders: i.e.,
	\begin{align*}
		\beta_1<_{\FF,\out(v_2)}\gamma \text{\ and \ } \alpha_1<_{\FF,\inn(v_3)}\beta_1 \text{\ and \ } \alpha_2<_{\FF,\out(v_3)}\beta_2.
	\end{align*}

	\begin{figure}
		\centering
		\def\svgscale{.38}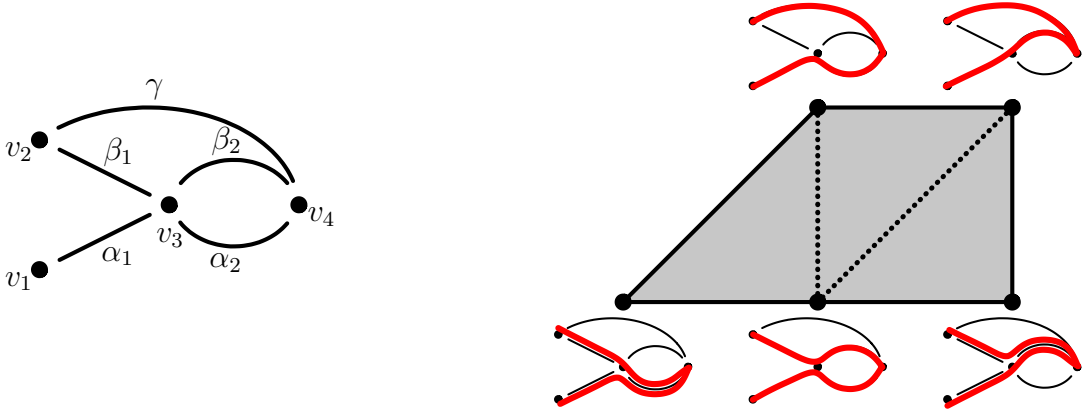
		\caption{A framed DAG with netflow vector $\a=(1,1,0,-2)$ and its framing-triangulated flow polytope.}
		\label{fig:cons}
	\end{figure}

	The flow polytope $\F_G(\a)$ is drawn on the right of Figure~\ref{fig:cons} with its integer points labelled by the five layerings of $(G,\a,\FF)$.
	The three maximal layering-simplices correspond to the three two-dimensional simplices of the framing triangulation of $\F_G(\a)$ indicated using dotted lines: for example, the leftmost simplex represents the maximal layering-simplex
	\[
		\big\{
			\{\alpha_1\alpha_2,\beta_1\alpha_2\},
			\{\alpha_1\alpha_2,\beta_1\beta_2\},
			\{\alpha_1\alpha_2,\gamma\}
		\big\}.
	\]
	The set $\big\{\{\alpha_1\alpha_2,\beta_1\alpha_2\},\{\alpha_1\beta_2,\gamma\}\big\}$ is not a layering-simplex because its underlying set of routes is not a route-clique: $\alpha_1\beta_2$ is incompatible with $\beta_1\alpha_2$.
	Now consider $L^1:=\{\alpha_1\alpha_2,\gamma\}$ and $L^2:=\{\alpha_1\beta_2,\beta_1\beta_2\}$; note that $L^1<_\src^+L^2$. The set $\routes(\{L^1,L^2\})$ is a route-clique. On the other hand, the $<_\src^+$-minimal layering using routes of $\routes(\{L^1,L^2\})$ is $\{\alpha_1\alpha_2,\beta_1\beta_2\}\not\in\{L^1,L^2\}$, so $\{L^1,L^2\}$ fails to be a layering-simplex by Definition~\ref{defn:layering-simplex}~(2).

	We will see later, phrased in the language of \emph{well-ordered} framed augmentations, that the layering-simplices of $(G,\a,\FF)$ may be calculated using a pairwise compatibility condition which is relatively simple compared to Definition~\ref{defn:layering-simplex} (see Examples~\ref{ex:augpoly} and~\ref{ex:augpolyg}).
\end{example}

Before proving the main result in conservationist generality, we need one technical proposition.

A \emph{bipartite graph} $\Gamma$ is an unoriented graph whose vertex set is partitioned $V=X\sqcup Y$ such that all edges of $\Gamma$ are incident to a vertex of $X$ and a vertex of $Y$. A graph is \emph{regular} if every vertex has the same degree.

\begin{lemma}\label{lem:reg}
	If $\Gamma$ is a regular bipartite graph, then $\Gamma$ has a perfect matching.
\end{lemma}
\begin{proof}
	Follows from Hall's perfect matching theorem.
\end{proof}

\begin{prop}\label{prop:tech}
	Let $(G,\a,\FF)$ be conservationist and let $F$ be an $x\times\a$-flow of $G$ for $x>0$ expressed as a route-clique combination
	\[
		F=\sum_{p\in\R}a_p\I(p).
	\]
	Then there exists a layering $L$ using only routes of $\R$.
\end{prop}
\begin{proof}
	Our first goal is to obtain a nonzero rational $\a$-flow $F_\rat$ of $G$ which is expressed as a route-clique combination using only routes of $\R$:
	\[F_\rat=\sum_{p\in\R}b_p\I(p).\]
	First, define $F':=\frac{F}{x}$ and note that $F'$ is an $\a$-flow of $G$.
	Then $F'=\sum_{p\in\R}\frac{a_p^{\R}}{x}\I(p)$ is the (unique by Proposition~\ref{prop:route-clique-combinations}) positive route-clique combination realizing $F'$.
	If $F$ is rational-valued, define $F_\rat:=F'$ and $b_p:=\frac{a_p^\R}{x}$ for $p\in\R$. Otherwise, if $F$ is irrational then we will find a rational $\a$-flow $F_\rat$ realized by a route-clique combination whose route-clique is $\R$.
	Recall that $\R$ is a clique of $(\overline G,\overline\FF)$; Corollary~\ref{cor:DKK} applied to the DKK-framed DAG $(\overline G,\overline\FF)$ then shows that the points $\{S_\a\I(p)\ :\ p\in\R\}$ are the vertices of a simplex $\D_{S_\a}(\R)$.
	Then $F'=\sum_{p\in\R}\frac{a_p^{\R}}{x}\I(p)$ presents $F'$ within the interior of this simplex because $a_p^{\R}>0$ for all $p\in\R$ and $\sum_{p\in\R}a_p=xS_\a$.
	We have shown that $F'\in\D_{S_\a}(\R)\cap\F_G(\a)$, so the polytope $P:=\D_{S_\a}(\R)\cap\F_G(\a)$ must be nonempty. Since $P$ is the intersection of two rational polyhedra, there must exist a rational point $F_{\rat}$ of $P$. Since $F_\rat\in\F_G(\a)$, the point $F_\rat$ is an $\a$-flow of $G$. Since $F_{\rat}\in\D_{S_\a}$, it must have a presentation as a convex combination of the vertices of this simplex
	\[
		F_{\rat}=\sum_{p\in\R}c_p(S_\a\I(p))=\sum_{p\in\R}(c_pS_\a)\I(p).
	\]
	Then define $b_p:=c_pS_\a$ for $p\in\R$.
	In all cases, we have a rational $\a$-flow $F_\rat$ of $G$ expressed as a route-clique combination $F_\rat=\sum_{p\in\R}b_p\I(p)$ using only routes of $\R$.

	Choose now $n\in\mathbb Z_{\geq1}$ such that $F_\int:=nF'_\rat$ is integer-valued. Setting $b'_p:=nb_p$ for $p\in\R$, we then have $F_\int=\sum_{p\in\R}b'_p\I(p)$ and hence $b'_p$ are all integer-valued by Theorem~\ref{thm:DKK}.
	Let $M$ be the multiset of routes where each $p\in\R$ appears with multiplicity $b'_p$ (note that we may have $b'_p=0$ for some $p\in\R$, in which case this route does not appear in $M$).
	Since $F_\int$ is an $n\times\a$-flow of $G$, there are $nS_\a$ elements of $M$; moreover, at each source $x$ there are $n\a_x=n$ elements of $M$ beginning at $x$ and for each sink $y$ there are $n|\a_y|$ elements of $M$ ending at $y$. Order these routes $p_{y,1},\dots,p_{y,n|\a_y|}$ arbitrarily, so that every route of $M$ has a unique description as some $p_{y,i}$ (where $y$ is a sink of $G$ and $i\in[|\a_y|]$).

	Let $X$ be the set of source vertices of $G$. Let
	\[Y:=\{y_i\ :\ y\text{ is a sink of }G\text{ and }i\in[|\a_y|] \}.\]
	Define now a bipartite graph $\Gamma$ with vertex set $X\sqcup Y$ and edge set
	\[
		\{e_{y,i}:t(p)\to h(p)_{i'}\ :\ p_{y,i}\in M\text{ and }i'\in[|\a_y|]\text{ with }i'\equiv i\text{ (mod $\a_y$)}.
		\}
	\]
	For any $x\in X$, there are $n$ elements of $M$ beginning at $x$, hence $n$ edges of $\Gamma$ beginning at $x$. For any $y_i\in Y$, there are $n$ elements of $[n|\a_y|]$ equivalent to $i$ modulo $|\a_y|$, hence $n$ edges of $\Gamma$ ending at $y_i$. Then Lemma~\ref{lem:reg} implies that there exists a perfect matching $N$ of $\Gamma$.
	Now define
	\[
		L:=\{p_{y,i}\ :\ e_{y,i}\in N\}.
	\]
	Because $N$ is a perfect matching, there is $\a_x=1$ route of $L$ beginning at every source $x$ of $G$. For every sink $y$ of $G$, there is one edge $e_{y,i}$ incident to $y_i$ for every choice $i\in|\a_y|$, hence a total of $|\a_y|$ routes $p_{y,i}$ of $L$ ending at $y$. Finally, all routes of $N$ are in $\mathcal R$ so $N$ is a route-clique. This completes the proof that $L$ is a layering using only routes of $\R$.
\end{proof}

We are now able to prove the main result of the article in the conservationist generality.

\begin{thm}\label{thm:triangulation}
	Let $(G,\a,\FF)$ be conservationist and let $F$ be an $x\times\a$-flow of $G$ for $x>0$. There is a unique positive layering-simplex combination
	\[
		F=\sum_{L\in\L}a_L\J(L)
	\]
	for $F$. Moreover, if $F$ is integer-valued, then all coefficients $a_L$ are integers.
\end{thm}
\begin{proof}
	Let $x>0$ and let $F$ be an $x\times\a$-flow of $G$. We will first algorithmically calculate a layering-simplex combination realizing $F$. 
	Use Proposition~\ref{prop:route-clique-combinations} to decompose $F$ as a route-clique combination
	\[F=\sum_{p\in\R}a_{p}^{\R}\I(p).\]

	Start by defining $F_1:=F$ and $x_1:=x$ and $\R_1:=\R$.
	Given $x_{j}>0$ and an $x_{j}\times\a$-flow $F_{j}$ realized by a positive route-clique combination using $\R_{j}\subseteq\R$, we proceed as follows.

	Proposition~\ref{prop:tech} shows that there exists at least one layering using only the routes of $\R_{j}$. Define $L^{j}:=\min_{<_\src^+}\left\{L\in\text{Layerings}(G,\a,\FF)\ :\ \routes(\{L\})\subseteq\R_{j}\right\}$.
	Define $a_{L^{j}}:=\min\{a_{p}^{\R_{j}}\ :\ p\in L^{j}\}$.

	Let $\R_{j+1}:=\R_{j}\backslash\{p\in L^{j}\ :\ a_p^{\R_{j}}=a_{L^{j}}\}$. For $p\in\R_{j+1}$, define the positive coefficient \[a_p^{\R_{j+1}}:= \begin{cases} a_p^{\R_{j}} & p\not\in L^{j}\\ a_p^{\R_{j}}-a_{L_{j}} & p\in L^{j}.\end{cases}\]
		Let $F_{j+1}:=F_{j}-a_{L^{j}}\J(L^{j})$ and $x_{j+1}:=x_{j}-a_{L_j}$. If $F_{j+1}$ is the zero function, terminate the algorithm. Otherwise, it is now immediate that
	\[F_{j+1}=\sum_{p\in\R_{j+1}}a_p^{\R_{j+1}}\I(p)\]
	is the positive route-clique combination for the $x_{j+1}\times\a$-flow $F_{j+1}$.

	This process must terminate because $\R_j\supsetneq\R_{j+1}$ as long as the algorithm continues. Say the algorithm returns $\L:=\{L^1,L^2,\dots,L^m\}$ for some $m\in\mathbb Z_{\geq1}$.
	We claim that $\L$ is a layering-simplex.

	Because the routes of $L^j$ come from $\R^j\subseteq\R$ for all $j\in[m]$, the set $\routes(\L)$ is a route-clique, showing condition (1) of Definition~\ref{defn:layering-simplex}.
	For any $i\in[m]$, the layering $L^j$ is defined to be the $<_\src^+$-minimal layering using routes of $\R^j=\routes(\{L^{j+1},\dots,L^m\})$, showing condition (2). 
	This ends the proof that $\L$ is a layering-simplex.
	Then
	\[
		F=\sum_{L\in\L}a_L\J(L)
	\]
	is the desired layering-simplex combination realizing $F$.

	Moreover, suppose that $F$ is integer-valued. Then all coefficients $a_p^{\R_1}$ are integers by the final sentence of Proposition~\ref{prop:route-clique-combinations}, hence $a_{L_1}$ is an integer, hence all coefficients $a_p^{\R_2}$ are integers, and so on. This shows that if $F$ is integer-valued, then all coefficients $a_L$ are integers.

	It remains to show that any $x\times\a$-flow $F$ is \emph{uniquely} realized as a positive layering-simplex combination. To this end, suppose $F=\sum_{M\in\M}a_M\J(M)$ is a positive layering-simplex combination for some $x\geq0$ and some $x\times\a$-flow $F$. We will complete the proof by showing that applying the above algorithm to $F$ retrieves the layering-simplex combination $\sum_{M\in\M}a_M\J(M)$.
	We will show this by induction on $|\M|$.

	The base case $|\M|=0$ is trivial, as $F$ must be the zero flow realized by the empty layering-simplex combination. We now proceed with the inductive step: Suppose $|\M|=m$ and suppose that we have shown this uniqueness property for layering-simplex combinations of length $\leq m-1$.
	Notate $\M=\{M^1,\dots,M^m\}$ as usual.
	Rearranging the layering-simplex combination
	\begin{align*}
		\label{eqn:a}
		F&=\sum_{M\in\M}a_M\J(M) \\
		&=\sum_{M\in\M}\sum_{p\in M}a_M\I(p)\\
		&=\sum_{p\in\routes(\{\M\})}\left(\sum_{M\in\M\ :\ p\in M}a_M\right)\I(p)
	\end{align*}
	gives the route-clique combination realizing $F$ using the route-clique $\routes(\{\M\})$.
	Set $\R:=\routes(\{\M\})$ and for $p\in\R$ set $a_p:=\sum_{M\in\M\ :\ p\in M}a_M$, so that $F=\sum_{p\in\R}a_p\I(p)$ is the route-clique combination realizing $F$.

	Then the first step of the algorithm retrieves \[L^1:=\min_{<_\src^+}\{M\in\textup{Layerings}(G,\a,\FF)\ :\ \routes(\{M\})\subseteq\routes(\M)\},\]
	which is equal to $M^1$ by condition (1) of Definition~\ref{defn:layering-simplex}.
	Moreover, Lemma~\ref{lem:route-deletion} shows that there exists a route $p\in M^1$ which is not in $\routes(\{M^2,\dots,M^m\})$. Then $a_p=a_{M^1}$. At the same time, $a_p=\min\{a_p\ :\ p\in M^1\}=\min\{a_p\ :\ p\in L^1\}=:a_{L^1}$. This shows that $a_{L^1}=a_{M^1}$. 

	The algorithm then continues anew applied to the flow $F_2:=F-a_{L^1}\J(L^1)$. This flow is already realized as a layering-simplex combination $F_2=\sum_{M\in\M\backslash\{M^1\}}a_M\J(M)$, so the induction hypothesis shows that the algorithm will agree with this layering-simplex combination continuing on. This completes the proof.
\end{proof}

In fact, Theorem~\ref{thm:triangulation} is a triangulation result on $\F_G(\a)$:

\begin{cor}\label{cor:triangulation}
	Let $(G,\a,\FF)$ be conservationist. The \emph{framing triangulation}
	\[
		\mathcal T(G,\a,\FF):=\{\D_1(\L)\ :\ \L\in\textup{Layerings}(G,\a,\FF)\}
	\]
	is a unimodular lattice triangulation of $\F_G(\a)$.
\end{cor}
\begin{proof}
	Recall the definition of a lattice triangulation from Definition~\ref{defn:triangulation}.
	First, note that uniqueness of layering-simplex combinations given in Theorem~\ref{thm:triangulation} shows that if $\L$ is a layering-simplex, then the set $\{\J(L)\ :\ L\in\L\}$ has no affine dependencies, hence $\D_1(\L)=\textup{conv}\{\J(L)\ :\ L\in\L\}$ is a simplex.
	This in turn implies that the set of faces of $\D_1(\L)$ is $\{\D_1(\L')\ :\ \L'\subseteq\L\}$, proving Condition (2) of Definition~\ref{defn:triangulation}.
	Condition (1) follows because every flow is realized as a layering-simplex combination by Theorem~\ref{thm:triangulation}.

	To show Condition (3), we claim that if $\K$ and $\L$ are layering-simplices then $\D_1(\K)\cap\D_1(\L)=\D_1(\K\cap\L)$.
	It is immediate by definition that $\D_1(\K)\cap\D_1(\L)\supseteq\D_1(\K\cap\L)$. On the other hand, take any $F\in\D_1(\K)\cap\D_1(\L)$ and write it as a positive layering-simplex combination $F=\sum_{M\in\M}a_M\J(M)$.
	By uniqueness of Theorem~\ref{thm:triangulation}, we must have $\M\subseteq\K$ and $\M\subseteq\L$, hence $F\in\D_1(\K\cap\L)$.
	This completes the proof that $\mathcal T(G,\a,\FF)$ is a lattice triangulation of $\F_G(\a)$.

	The final sentence of Theorem~\ref{defn:triangulation} shows that for any layering $\L$, the vertices of $\D_1(\L)$ form a $\mathbb Z$-basis of their linear span, hence that $\D_1(\L)$ is unimodular.
\end{proof}

We finally make the following remark, which follows from Corollary~\ref{cor:triangulation} because all maximal cells of a triangulation of a polytope $P$ are of dimension $\dim(P)$.
\begin{cor}\label{cor:card}
	The cardinality of every maximal layering-simplex of a conservationist $(G,\a,\FF)$ is $\dim(\F_G(\a))+1$ (with the convention that an empty polytope is of dimension $-1$).
\end{cor}

Recall Example~\ref{ex:consexample}, which features the framing triangulation on the right of the relevant Figure~\ref{fig:cons}.
Note that the layering-simplices of Example~\ref{ex:consexample} may be defined from a pairwise compatibility conditions of the layerings, similar to the DKK-framing triangulations of the unit case. The following example does not have this property.

\begin{example}\label{ex:k33}
	\begin{figure}
		\centering
		\def\svgscale{.38}%% Creator: Inkscape 1.2.2 (b0a84865, 2022-12-01), www.inkscape.org
%% Accompanies image file 'im_k33.pdf' (pdf, eps, ps)
\begingroup%
  \makeatletter%
  \providecommand\color[2][]{%
    \errmessage{(Inkscape) Color is used for the text in Inkscape, but the package 'color.sty' is not loaded}%
    \renewcommand\color[2][]{}%
  }%
  \providecommand\transparent[1]{%
    \errmessage{(Inkscape) Transparency is used (non-zero) for the text in Inkscape, but the package 'transparent.sty' is not loaded}%
    \renewcommand\transparent[1]{}%
  }%
  \providecommand\rotatebox[2]{#2}%
  \newcommand*\fsize{\dimexpr\f@size pt\relax}%
  \newcommand*\lineheight[1]{\fontsize{\fsize}{#1\fsize}\selectfont}%
  \ifx\svgwidth\undefined%
    \setlength{\unitlength}{246.27001953bp}%
    \ifx\svgscale\undefined%
      \relax%
    \else%
      \setlength{\unitlength}{\unitlength * \real{\svgscale}}%
    \fi%
  \else%
    \setlength{\unitlength}{\svgwidth}%
  \fi%
  \global\let\svgwidth\undefined%
  \global\let\svgscale\undefined%
  \makeatother%
  \begin{picture}(1,0.6447679)%
    \lineheight{1}%
    \setlength\tabcolsep{0pt}%
    \put(0,0){\includegraphics[width=\unitlength,page=1]{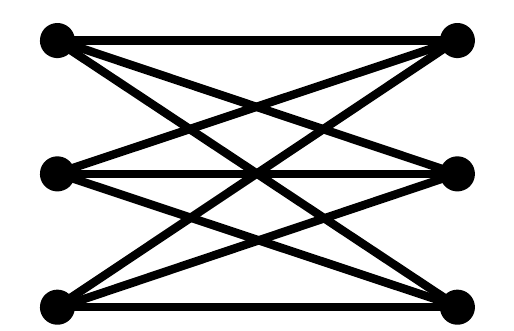}}%
    \put(0.00980631,0.11520689){\color[rgb]{0,0,0}\makebox(0,0)[lt]{\lineheight{1.25}\smash{\begin{tabular}[t]{l}$s_1$\end{tabular}}}}%
    \put(0.00980631,0.3750883){\color[rgb]{0,0,0}\makebox(0,0)[lt]{\lineheight{1.25}\smash{\begin{tabular}[t]{l}$s_2$\end{tabular}}}}%
    \put(0.00980631,0.63496565){\color[rgb]{0,0,0}\makebox(0,0)[lt]{\lineheight{1.25}\smash{\begin{tabular}[t]{l}$s_3$\end{tabular}}}}%
    \put(0.87952647,0.63496159){\color[rgb]{0,0,0}\makebox(0,0)[lt]{\lineheight{1.25}\smash{\begin{tabular}[t]{l}$t_3$\end{tabular}}}}%
    \put(0.87952647,0.3750883){\color[rgb]{0,0,0}\makebox(0,0)[lt]{\lineheight{1.25}\smash{\begin{tabular}[t]{l}$t_2$\end{tabular}}}}%
    \put(0.87952647,0.11521095){\color[rgb]{0,0,0}\makebox(0,0)[lt]{\lineheight{1.25}\smash{\begin{tabular}[t]{l}$t_1$\end{tabular}}}}%
  \end{picture}%
\endgroup%

		\caption{The complete bipartite graph $K_{3,3}$.}
		\label{fig:k33}
	\end{figure}
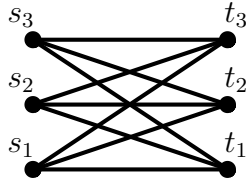

	Figure~\ref{fig:k33} shows the complete bipartite graph $K_{3,3}$, with edges oriented left to right.
	We supply it with the netflow vector $\a$ giving every source netflow $1$ and every sink netflow $-1$. Observe that $(K_{3,3},\a)$ is conservationist.
	Define a framing $\FF$ on $(K_{3,3},\a)$ by inducing each framing order $<_{\FF,\out(s_i)}$ from the embedding and ordering $s_1<_{\FF,\src}s_2<_{\FF,\src}s_3$ (Remark~\ref{remk:embedding}).
	Layerings of $(K_{3,3},\a)$ correspond to perfect matchings of $K_{3,3}$, of which there are six.
	We will index the layerings as $L(1)<_\src^+ \dots <_\src^+L(1)$.
	Figure~\ref{fig:k33order} shows these layerings ordered left-to-right by $<_{\src}^+$:

	\begin{figure}[H]
		\centering
		\def\svgscale{.38}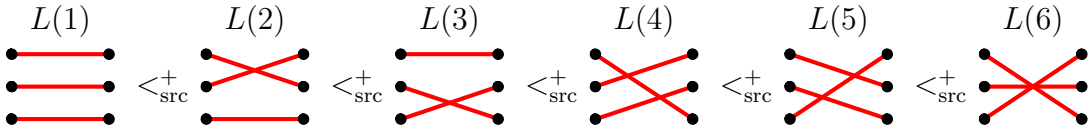
		\caption{The order $<_\src^+$ of all six layerings of $(K_{3,3},\a,\FF)$.}
		\label{fig:k33order}
	\end{figure}

	Let $\mathcal L=\{L^1,\dots,L^m\}$ be a set of layerings of $(K_{3,3},\a,\FF)$. We claim that $\mathcal L$ is a layering-simplex if and only if $\{L(2),L(3),L(6)\}\not\subseteq\mathcal L$.

	Because there are no internal vertices of $G$, every set of routes is a route-clique, so Definition~\ref{defn:layering-simplex}~(1) holds for every set of layerings. Then $\L$ is a layering-simplex if and only if it satisfies Definition~\ref{defn:layering-simplex}~(2).

	Consider Definition~\ref{defn:layering-simplex}~(2).
	Observe that for $i\neq2$ we have 
	\[L(i)=\min_{<_\src^+}\left\{L\in\text{Layerings}(K_{3,3},\a,\FF)\ :\ \routes(\{L\})\subseteq\routes(\{L(i),\dots,L(6)\})\right\},\]
	so Condition (2) may fail only when $\L^i=L(2)$. 
	Hence, if $\L$ does not contain $L(2)$ then it is a layering-simplex; suppose now that $\L$ contains $L(2)$. Write $L^i=L(2)$ for $i\in\{1,2\}$.
	If $\L$ contains $\{L(2),L(3),L(6)\}$, where $L^i=L(2)$, then
	\[
		\min_{<_\src^+}\left\{L\in\text{Layerings}(K_{3,3},\a,\FF)\ :\ \routes(\{L\})\subseteq\routes(\{L^i,\dots,L^m\})\right\}=L(1)\neq L^i
	\]
	and $\L$ is not a layering-simplex by Definition~\ref{defn:layering-simplex}~(2). On the other hand, if $\L$ does not contain $\{L(2),L(3),L(6)\}$ then $\routes(\{L^i,\dots,L^m\})$ is missing at least one route of $L(1)$ and
	\[
		\min_{<_\src^+}\left\{L\in\text{Layerings}(K_{3,3},\a,\FF)\ :\ \routes(\{L\})\subseteq\routes(\{L^i,\dots,L^m\})\right\}=L(2)=L^i,
	\]
	proving that $\L$ satisfies Definition~\ref{defn:layering-simplex}~(2) in all positions and hence is a layering-simplex.
	This completes the proof that a set $\L$ of layerings of $(K_{3,3},\a,\FF)$ is a layering-simplex if and only if $\{L(2),L(3),L(6)\}\not\subseteq\L$.

	In particular, there are three maximal layering-simplices of $(K_{3,3},\a,\FF)$, each of which misses exactly one element of $\{L(2),L(3),L(6)\}$. This shows that the framing triangulation of $\F_{K_{3,3}}(\a)$ induced by $\FF$ has three maximal simplices, hence $\F_{K_{3,3}}(\a)$ has normalized volume 3.

	We now observe two properties of this example which cannot appear in the unit case.
	\begin{enumerate}
		\item The set of any two layerings $\{L(i),L(j)\}$ is a layering-simplex, but the set of all layerings is not a layering-simplex. This means that the layering-simplex complex cannot be induced by a pairwise compatibility condition as in the unit DKK-framed case; the simplicial complex of layering-simplices is not flag. In~\cite{PLAN}, it was observed that no triangulation of $\F_{K_{3,3}}(\a)$ is induced by a flag complex on integer points.
		\item The dual graph of this framing triangulation is the complete graph on three vertices $K_3$. Notably, this graph is not the Hasse diagram of a partially ordered set, providing an obstacle to any extension of Ceballos and von Bell's theory of framing lattices~\cite{vBC} to our general setting of framing triangulations.
	\end{enumerate}
	In Section~\ref{sec:well-ordered}, we will provide (in the framed augmentation setting) conditions under which we expect more properties from the DKK-framing setting to hold.
\end{example}

\section{Framing Triangulations in Full Generality via Framed Augmentations}
\label{sec:aug}

The previous section gave a theory of framings and framing triangulations for DAGs with conservationist netflow vectors. 
We will now use this to give a theory of framing triangulations for an arbitrary DAG $G$ with (not necessarily conservationist) netflow vector $\a$ adding some extra edges to get an \emph{augmentation} $(\hat G,\hat\a)$ with a straightforward integral equivalence $\F_{\hat G}(\hat\a)\ucong\F_G(\a)$. We will then define framings on augmentations which will give a framing triangulation analogous to the conservationist case.
We will begin by giving the main result in this full generality with an example before moving on to the proof.

\subsection{Definitions and Results}

{We now define augmentations of an arbitrary graph $G$ and netflow vector $\a$, which are similar to the augmented graphs defined in~\cite{DHY}. We will then define variants of framings, layerings, layering-simplices, and framing triangulations in the setting of augmentations.}

\begin{defn}\label{defn:cover}
	Let $G=(V,E)$ be a graph with netflow vector $\a$. We say that a graph $\hat G=(\hat V,\hat E)$ with netflow vector $\hat\a$ is an \emph{augmentation} of $(G,\a)$ if $\hat V=V\sqcup V_X\sqcup V_Y$ and $\hat E=E\sqcup E_X\sqcup E_Y$ where
	\begin{enumerate}
		\item $E_X$ is a set of $S_\a$ \emph{inflow edges} from $V_X$ to $\{i\in V\ :\ \a_i>0\}$ such that $|\{x\in E_X\ :\ h(x)=i\}|=\a_i$ for all $i\in V$ with $\a_i>0$,
		\item $E_Y$ is a set of \emph{outflow edges} from $V_Y$ to $\{i\in V\ :\ \a_i<0\}$ such that $1\leq|\{y\in E_Y\ :\ t(y)=i\}|\leq|\a_i|$ for all $i\in V$ with $\a_i<0$,
		\item every vertex of $V_X$ (resp. $V_Y$) is incident to precisely one edge of $E_X$ (resp. $E_Y$), and
		\item the netflow vector $\hat\a$ sends all vertices of $V$ to 0, all vertices of $V_X$ to $1$, and all vertices of $V_Y$ to negative integers subject to the condition
			\begin{align*}
				\a_i&=\sum_{y\in E_Y\ :\ t(y)=i}\hat\a_{h(y)}\text{ for all }i\in V\text{ with }\a_i<0.
			\end{align*}
	\end{enumerate}
\end{defn}

When we have chosen $(G,\a)$ as well as an augmentation $(\hat G,\hat\a)$, for a vertex $i\in V$ we will use $\inn(i)$ to refer to the edges of $G$ with head $i$ and we will use $\hinn(i)$ to refer to the edges of $\hat G$ with head $i$. Similarly we will use $\out(i)$ and $\hout(i)$.

See Figure~\ref{fig:aug} for an example of augmentations.

\begin{figure}
	\centering
	\def\svgscale{.38}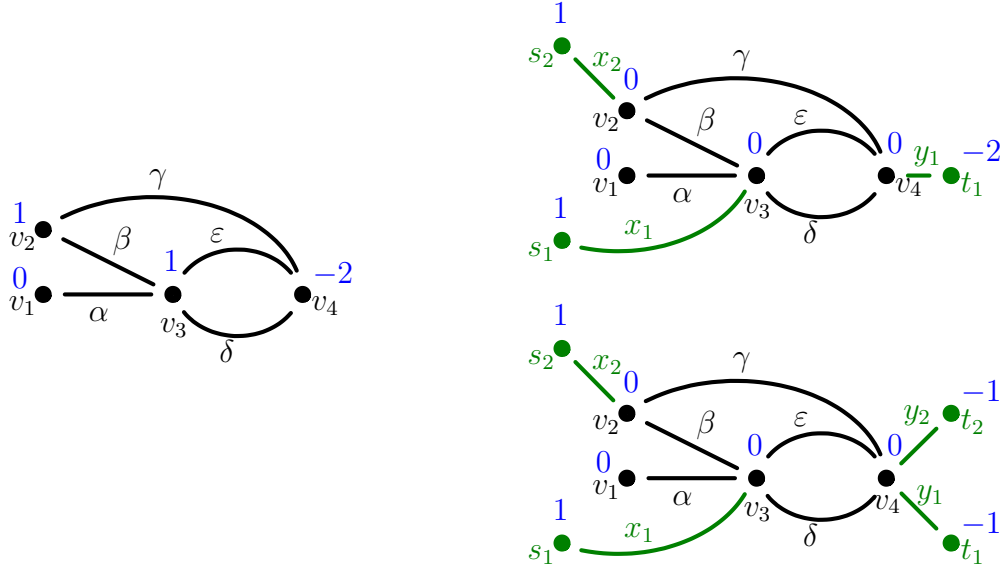
	\caption{On the left is a DAG $G$ with netflow vector $\a$, and on the right are the two possible augmentations of $(G,\a)$ with new vertices and edges in green.}
	\label{fig:aug}
\end{figure}

\begin{lemma}\label{lem:coverint}
	Let $G$ be a DAG with netflow vector $\a$ and let $(\hat G,\hat\a)$ be an augmentation of $(G,\a)$. The projection map $F\mapsto F|_E$ is an integral equivalence from $\F_{\hat G}(\hat\a)$ to $\F_G(\a)$.
\end{lemma}
\begin{proof}
	Follows immediately from the definitions. Alternatively, note that given an augmentation $(\hat G,\hat\a)$, one retrieves $(G,\a)$ by performing contractions along every inflow edge and outflow edge.
\end{proof}

\begin{defn}\label{defn:framingB}
	Let $G$ be a DAG with netflow vector $\a$ and let $(\hat G,\hat\a)$ be an augmentation of $(G,\a)$. A \emph{framing} $\hat\FF$ of the augmentation $(\hat G,\hat\a)$ is the data of
	\begin{enumerate}
		\item for each vertex of $V$, a total order $<_{\hat\FF,\hout(i)}$ on the outgoing edges $\hout(i)$, 
		\item for each vertex of $V$, a total order $<_{\hat\FF,\hinn(i)}$ on the incoming edges $\hinn(i)$, and
		\item a total order $<_{\hat\FF,\src}$ on $V_X$.
	\end{enumerate}
	We will notate the vertices of $V_X$ as $\s_1<_{\hat\FF,\src}\dots<_{\hat\FF,\src}\s_{S_\a}$.
	In our figures, we will depict framings on augmentations through the embeddings and labels on vertices of $V_X$ (see Remark~\ref{remk:embedding}).
	We say that $(\hat G,\hat\a,\hat\FF)$ is a \emph{framed augmentation} of $(G,\a)$.
\end{defn}

Note that we only need orders $<_{\hat\FF,\hout(i)}$ and $<_{\hat\FF,\hinn(i)}$ at the vertices of $V$, since each vertex of $V_X$ and $V_Y$ is incident to only one edge.

We now define pairwise compatibility of routes of an augmentation after Definitions~\ref{defn:compat1} and~\ref{defn:compat2} in the conservationist case.

\begin{defn}\label{defn:compat1B}
	Let $(\hat G,\hat\a,\hat\FF)$ be a framed augmentation of $(G,\a)$. Define as in Definition~\ref{defn:compat1} the \emph{post-$i$ order} $<_i^+$ on the paths from a vertex $i$ of $\hat G$ to a sink, and the \emph{pre-$i$ order} $<_i^-$ on the paths from a source to a vertex of $\hat G$. If $p$ and $q$ are routes with a shared vertex $i\in\hat V$ then we may write $p<_i^-q$, $p>_i^-q$, or $p=_i^+q$ as in Definition~\ref{defn:compat1}; similarly we have the total preorder $<_i^+$ on routes passing through $i$.

	If $p$ and $q$ are routes $i$ is a vertex contained in both $p$ and $q$, then the routes $p$ and $q$ are \emph{incompatible} at $i$ if, without loss of generality, $p<_i^-q$ and $q<_i^+p$. Two routes $p$ and $q$ are \emph{incompatible} if they are incompatible at any shared vertex, otherwise they are \emph{compatible}.
	A \emph{route-clique} of $(\hat G,\hat\a,\hat\FF)$ is a set of pairwise compatible routes.
\end{defn}

We now define layerings (Definition~\ref{defn:layering}), layering-simplices (Definition~\ref{defn:layering-simplex}), and polyhedral layering-simplices of a framed augmentation of $(G,\a)$.

\begin{defn}[{Definition~\ref{defn:layering}}]
	\label{defn:layeringB}
	Let $(\hat G,\hat\a,\hat\FF)$ be a framed augmentation of $(G,\a)$.
	A \emph{layering} $L$ of $(\hat G,\hat\a,\hat\FF)$ is a set of routes such that
	\begin{enumerate}
		\item for every source vertex $i$ of $\hat E$, there is precisely $\hat\a_i=1$ route of $L$ beginning at $i$, and
		\item for every sink vertex $j$ of $\hat E$, there are precisely $|\hat\a_j|$ routes of $L$ ending at $j$.
	\end{enumerate}
	We will index a layering $L$ as $L=\{L_1,\dots,L_{S_\a}\}$, where for $i\in[S_\a]$ the route $L_i$ begins at $\s_i$.
	The set of layerings of $(\hat G,\hat\a,\hat\FF)$ is notated $\textup{Layerings}(\hat G,\hat\a,\hat\FF)$.

	Let $L$ and $M$ be distinct layerings of $(\hat G,\hat\a,\hat\FF)$. Choose $i\in[S_\a]$ maximal such that $L_i\neq M_i$ and suppose without loss of generality that $L_i<_{\s_i}+M_i$. We say that $L<_\src^+M$. It is immediate that $<_\src^+$ is a total order on layerings, which we call the \emph{post-source order} on layerings.
\end{defn}

\begin{lemma}\label{lem:aug-layer-xy}
	If $L$ is a layering of a framed augmentation $(\hat G,\hat\a,\hat\FF)$, then all routes of $L$ begin with an edge of $E_X$ and end with an edge of $E_Y$.
\end{lemma}
\begin{proof}
	We will show that no route begins with a vertex of $\hat V\backslash V_X=V\cup V_Y$ and that no route ends with a vertex of $\hat V\backslash V_Y=V\cup V_X$.
	Choose a vertex $i\in V\cup V_Y$. Then $\hat\a_i=0$ by definition of $\hat\a_i$, hence $L$ contains $\hat\a_i=0$ routes beginning with $V$. One argues symmetrically that there are no routes ending with a vertex of $V\cup V_X$.
\end{proof}

If $\L=\{L^1,\dots,L^m\}$ is a set of layerings, let
$\routes(\L):=\{L^i_j\ :\ i\in[m],\ j\in[S_\a]\}=\cup_{i=1}^mL^i$ be the underlying set of routes of $\L$.

\begin{defn}[{Definition~\ref{defn:layering-simplex}}]
	\label{defn:layering-simplexB}
	Let $(\hat G,\hat\a,\hat\FF)$ be a framed augmentation of $(G,\a)$.
	Let $\L=\{L^1,\dots,L^m\}$ be a set of layerings of $(\hat G,\hat\a,\hat\FF)$ indexed so that $L^1<_\src^+\dots<_\src^+L^m$. We say that $\L$ is a \emph{layering-simplex} of $(G,\a,\FF)$ if
	\begin{enumerate}
		\item the set $\routes(\L)$ is a route-clique,
		\item for any $i\in[m]$, we have 
			\[L^i=\min_{<_\src^+}\left\{L\in\text{Layerings}(\hat G,\hat\a,\hat\FF)\ :\ \routes(\{L\})\subseteq\routes(\{L^{i},\dots,L^m\})\right\}\textup{, and}\]
	\end{enumerate}
	If $\L$ is a layering-simplex, then define the \emph{polyhedral layering-simplex}
	\[
		\D_1(\L):=\text{conv}
		\{
			\hat\J(L)\ :\ L\in\L
		\}.
	\]
\end{defn}

We will finally use layering-simplices to get at simplices of a triangulation of $\F_G(\a)$.
To do this, we will associate a route
$p$ of an augmentation $(\hat G,\hat\a)$ of $(G,\a)$ with its \emph{deaugmented indicator vector} $\check\I(p):E\mapsto\mathbb R_{\geq0}$ defined to send every edge of $E$ used by $p$ to $1$ and all other edges of $E$ to 0. Note that this is merely the projection of the usual indicator vector $\I(p)\in\mathbb R^{\hat E}$ to $\mathbb R^E$ (recall that this projection is an integral equivalence from $\F_{\hat G}(\hat\a)$ to $\F_G(\a)$ by Lemma~\ref{lem:coverint}).
Accordingly, given a layering $L$ of $(\hat G,\hat\a,\hat\FF)$ we define its \emph{deaugmented indicator vector}
\[
	\check{\J}(L):=\sum_{p\in L}\check\I(p).
\]
Given a layering-simplex $\L$ of $(\hat G,\hat\a,\hat\FF)$, define  its \emph{deaugmented polyhedral layering-simplex}
\[
	\check{\D}_1(\L):=\text{conv}\{\check\J(L)\ :\ L\in\L\}.
\]
We now phrase the variant of Corollary~\ref{cor:triangulation} for general DAGs via framed augmentations.

\begin{thm}[{Corollary~\ref{cor:triangulation}}]
	\label{thm:main-aug}
	Let $(G,\a)$ be a framed DAG and let $(\hat G,\hat\a,\hat\FF)$ be a framed augmentation of $(G,\a)$. The \emph{framing triangulation}
	\[
		\mathcal T(\hat G,\hat\a,\hat\FF):=\{\check\D_1(\L)\ :\ \L\in\textup{Layerings}(\hat G,\hat\a,\hat\FF)\}
	\]
	is a unimodular lattice triangulation of $\F_G(\a)$.
\end{thm}

Similarly, to Corollary~\ref{cor:card} in the conservationist case, the following follows because all maximal cells of a triangulation of a polytope $P$ are of dimension $\dim(P)$.
\begin{cor}\label{cor:cardB}
	The cardinality of every maximal layering-simplex of a framed augmentation $(\hat G,\hat\a,\hat\FF)$ of $(G,\a)$ is $\dim(\F_G(\a))+1$.
\end{cor}

\begin{example}\label{ex:augpoly}
	Figure~\ref{fig:augpoly} shows one of the augmentations $(\hat G,\hat\a)$ from Figure~\ref{fig:aug}. Equip this augmentation with the framing $\hat\FF$ induced by the embedding (Remark~\ref{remk:embedding}). The right shows the flow polytope $\F_G(\a)$ with its integer points labelled by layerings, and the framing triangulation is drawn using dotted lines. 
	The combinatorics of layerings of $(\hat G,\hat\a,\hat\FF)$ is the same as that discussed in Example~\ref{ex:consexample} (discussing Figure~\ref{fig:cons}), so we will not repeat it here.
	\begin{figure}
		\centering
		\def\svgscale{.38}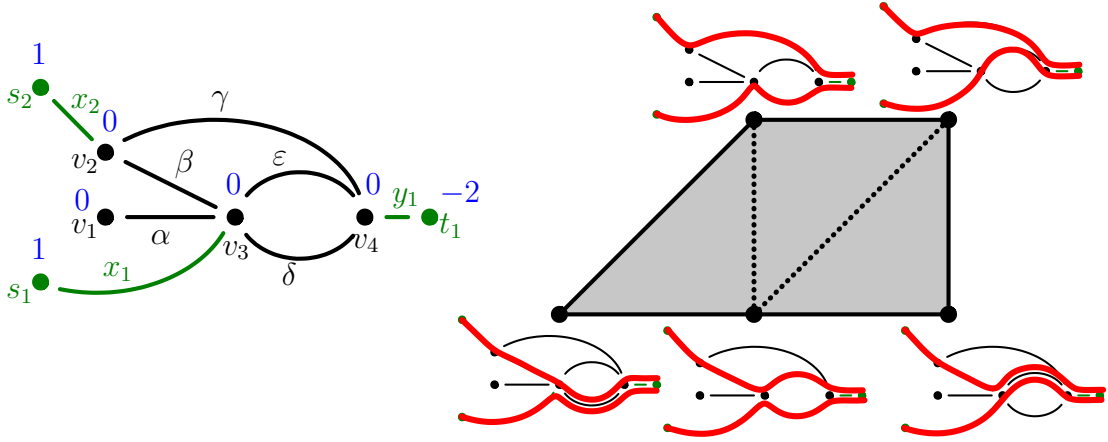
		\caption{The framing triangulation of an augmentation from Figure~\ref{fig:aug}.}
		\label{fig:augpoly}
	\end{figure}
\end{example}

We finally observe that if $G$ is a DAG with conservationist netflow vector $\a$, then the framing triangulations of $\F_G(\a)$ via the conservationist setting of Corollary~\ref{cor:triangulation} are the same as the framing triangulations via the augmentation setting of Theorem~\ref{thm:main-aug}:

\begin{prop}\label{prop:same}
	Let $G$ be a DAG with conservationist netflow vector $\a$. Then
	\begin{align*}
		&\{\mathcal T(G,\a,\FF)\ :\ \FF\text{ is a framing of the conservationist }(G,\a)\}
		\\
		=&\{\mathcal T(\hat G,\hat\a,\hat\FF)\ :\ (\hat G,\hat\a,\hat\FF)\text{ is a framed augmentation of }(G,\a)\}.
	\end{align*}
\end{prop}
In other words, the framing triangulations from framings of a conservationist $(G,\a)$ are the same as the framing triangulations from framed augmentations of $(G,\a)$.
\begin{proof}
	Let $(\hat G,\hat\a,\hat\FF)$ be a framed augmentation of $(G,\a)$. Define a framing $\FF$ on $(G,\a)$ in the sense of Definition~\ref{defn:framing} such that $<_{\FF,\inn(i)}=<_{\hat\FF,\hinn(i)}$ for all internal vertices of $G$ and $<_{\FF,\out(i)}=<_{\hat\FF,\hout(i)}$ for all non-sink vertices, and such that the order $<_{\FF,\src}$ is induced from the order $<_{\hat\FF,\hat\src}$ via identifying each outflow vertex with the unique source of $G$ it is adjacent to. It is then immediate by the definitions that $\mathcal T(G,\a,\FF)=\mathcal T(\hat G,\hat\a,\hat\FF)$.

	On the other hand, let $\FF$ be a framing of $(G,\a)$ in the sense of Definition~\ref{defn:framing}.
	Since $(G,\a)$ is conservationist, the sinks of $G$ are precisely those vertices with negative netflow.
	Choose the augmentation $(\hat G,\hat\a)$ such that every sink has exactly one outflow edge.
	If $i$ is a non-sink vertex of $G$, define $<_{\hat\FF,\hout(i)}:=<_{\FF,\out(i)}$; if $i$ is a sink of $G$ then define $<_{\hat\FF,\hout(i)}$ as the trivial order on the single outgoing edge of $i$ in $\hat G$. Similarly define the orders $<_{\hat\FF,\hinn(i)}$.
	As above, the order $<_{\FF,\src}$ induces a total order $<_{\hat\FF,\hat\src}$ via identifying each outflow vertex with the unique source of $G$ it is adjacent to. Then $\hat\FF$ is a framing of the augmentation $(\hat G,\hat\a)$ in the sense of Definition~\ref{defn:framingB} and it is immediate that $\mathcal T(G,\a,\FF)=\mathcal T(\hat G,\hat\a,\hat\FF)$.
\end{proof}

The rest of this section will be used to prove Theorem~\ref{thm:main-aug} by using Corollary~\ref{cor:triangulation}.

\subsection{Proving Theorem~\ref{thm:main-aug}}

Let $G=(V,E)$ be a DAG with netflow vector $\a$. We say that an edge $e\in E$ is \emph{flow-supporting} if there exists an $\a$-flow $F$ with $F(e)>0$. 
We say that $(G,\a)$ is \emph{flow-supporting} if every edge of $G$ is flow-supporting.

Let the \emph{flow-supporting subgraph} of an arbitrary $(G,\a)$ be the tuple $(G^\fs,\a^\fs)$, where $G^\fs=(V^\fs,E^\fs)$ is the subgraph of $G$ induced by the flow-supporting edges of $G$ and $\a^\fs=\a|_{V^\fs}$.

\begin{lemma}\label{lem:fs}
	If $F\in\F_G(\a)$, then the projection $F^\fs$ of $F\in\mathbb R^E$ to $\mathbb R^{E^\fs}$ is in $\F_{G^\fs}(\a^\fs)$, and this projection $(-)^\fs$ is an integral equivalence from $\F_G(\a)$ to $\F_{G^\fs}(\a^\fs)$.
\end{lemma}
\begin{proof}
	Immediate from the definitions.
\end{proof}

\begin{lemma}\label{lem:fs-is-c0}
	Every source $i$ of $V^\fs$ has positive netflow $\a^\fs_i>0$ and every sink $j$ of $V^\fs$ has negative netflow $\a^\fs_j<0$.
\end{lemma}
\begin{proof}
	Let $i$ be a source of $V^\fs$. Since $G^\fs$ is an edge-induced subgraph of $G$, there must be at least one edge $e$ incident to $i$; since $i$ is a source, we have $t(e)=i$. 
	By definition of $G^\fs$, the edge $i$ is flow-supporting in $(G,\a)$, hence there exists a flow $F\in\F_\a(G)$ rating $F(e)>0$. Then $F^\fs(e)>0$, forcing $\a^\fs_i>0$.
	A symmetric argument proves that sinks must have negative netflow.
\end{proof}

\begin{lemma}\label{lem:aug-cons}
	The netflow vector $\hat\a^\fs$ is a conservationist netflow vector of $\hat G^\fs$.
\end{lemma}
\begin{proof}
	We will first show that $V_X$ is the set of sources of $\hat G^\fs$. It is clear by definition that every vertex of $V_X$ is a source of $\hat G^\fs$. On the other hand, Lemma~\ref{lem:fs} implies that every source $i$ of $\hat G^\fs$ must have positive netflow $\hat\a_i>0$, hence $i\in V_X$ and $\hat\a_i=1$. Similarly, $V_Y$ is precisely the set of sinks of $\hat G^\fs$. Then $\hat\a^\fs$ must be conservationist because it is positive precisely on the source set $V_X$ and negative precisely on the sink set $V_Y$.
\end{proof}

\begin{lemma}\label{lem:aug-uni}
	If $F\in\F_{\hat G^\fs}(\hat\a^\fs)$, define 
	\begin{align*}
		\check{F}:E&\to\mathbb R_{\geq0} \\
		e&\mapsto
			\begin{cases}
				F(e) & e\in \hat E^\fs \\
				0 & e\not\in \hat E^\fs.
			\end{cases}
	\end{align*}
	Then the map $F\mapsto\check{F}$ is an integral equivalence from $\F_{\hat G^\fs}(\hat\a^\fs)$ to $\F_G(\a)$.
\end{lemma}
\begin{proof}
	Observe that the map $F\mapsto\check F$ is the composition of the inverse of the integral equivalence $(-)^\fs:\F_{\hat G}(\hat\a)\to\F_{\hat G^\fs}(\hat\a^\fs)$ (Lemma~\ref{lem:fs}) with the integral equivalence
	$(-)|_E\ :\ \F_{\hat G}(\hat\a)\to\F_G(\a)$
	(Lemma~\ref{lem:coverint}).
\end{proof}

We have finally proven that arbitrary integer flow polytopes may be reduced to the conservationist case, up to integral equivalence.

\begin{cor}\label{cor:everything-is-conservationist}
	Let $G$ be a DAG with arbitrary integer netflow vector $\a$. There exists a DAG $G'$ with conservationist netflow vector $\a'$ such that $\F_G(\a)\ucong\F_{G'}(\a')$.
\end{cor}
\begin{proof}
	Choose any augmentation $(\hat G,\hat\a)$ of $(G,\a)$ and let $(G',\a'):=(\hat G^\fs,\hat\a^\fs)$. Lemma~\ref{lem:aug-uni} gives the desired integral equivalence.
\end{proof}

Given a framed augmentation $(\hat G,\hat\a,\hat\FF)$ of $(G,\a)$, Lemma~\ref{lem:aug-cons} shows that $(\hat G^\fs,\hat\a^\fs)$ is conservationist. Moreover, by restricting the orders of the framing $\hat\FF$ to $\hat E^\fs$ one obtains a framing $\hat\FF^\fs$ of $(\hat G^\fs,\hat\a^\fs)$.
In fact, the layerings and layering-simplices of the conservationist $(\hat G^\fs,\hat\a^\fs,\hat\FF^\fs)$ correspond naturally to the layerings and layering-simplices of $(\hat G,\hat\a,\hat\FF)$:

\begin{prop}\label{prop:f}
	A set of routes of $\hat G$ is a layering-simplex of the augmentation $(\hat G,\hat\a,\hat\FF)$ (as in Definition~\ref{defn:layering-simplexB}) if and only if it is a layering-simplex of the conservationist $(\hat G^\fs,\hat\a^\fs,\hat\FF^\fs)$ (as in Definition~\ref{defn:layering-simplex}).
	Moreover, for a layering-simplex $\L$ of $(\hat G,\hat\a,\hat\FF)$, the integral equivalence $F\mapsto\check F$ of Lemma~\ref{lem:aug-uni} sends the polyhedral layering-simplex $\D_1(\L)$ to the deaugmented polyhedral layering-simplex $\check{\D}_1(\L)$.
\end{prop}
\begin{proof}
	It is immediate that a set of routes of $\hat G^\fs$ forms a layering-simplex of the conservationist $(\hat G^\fs,\hat\a^\fs,\hat\FF^\fs)$ if and only if it forms a layering-simplex of $(\hat G,\hat\a,\hat\FF)$. To show the first statement of the proposition, it remains only to show that a layering-simplex $\L$ of $(\hat G,\hat\a,\hat\FF)$ must be a set of routes of $\hat G^\fs$ (i.e., it cannot contain a route of $\hat G^\fs$ using an edge which is not flow-supporting).
	This is immediate upon verifying that
	\[
		\sum_{L\in\L}\sum_{p\in L}\I(p)
	\]
	is an $\hat\a$-flow of $\hat G$, hence all edges used by a route of $L$ are flow-supporting.
	The final statement of the proposition follows upon checking that $\check{\I(p)}=\check{\I}(p)$ for a route $p$ of $\hat G$.
\end{proof}

We are finally able to prove the main result of this section.

\begin{proof}[{Proof of Theorem~\ref{thm:main-aug}}]
	Proposition~\ref{prop:f} shows that the integral equivalence $F\mapsto\check F$ from $\F_{\hat G^\fs}(\hat\a^\fs)$ to $\F_G(\a)$ sends the framing triangulation $\mathcal T(\hat G^\fs,\hat\a^\fs,\hat\FF^\fs)$ (as defined in Corollary~\ref{cor:triangulation}) to $\mathcal T(\hat G,\hat\a,\hat\FF)$ as given in Theorem~\ref{thm:main-aug}, hence the latter must itself be a unimodular triangulation.
\end{proof}

\section{Well-ordered Framed Augmentations}
\label{sec:well-ordered}

In Example~\ref{ex:k33} we observed a framed conservationist DAG with two troubling properties: its layering-simplex complex is not a flag complex, and the dual graph of its framing triangulation is not the Hasse diagram of a poset.
In this section we isolate a subclass of ``well-ordered'' framed augmentations whose layering-simplex complex may be verified by a simple pairwise compatibility condition similar to the unit case of Danilov, Karzanov, and Koshevoy.
{We will show that all theories of framing triangulations currently existing in the literature (namely, the unit flow polytope case of Danilov, Karzanov, and Koshevoy~\cite{DKK}, the single-negative-netflow case of Gonz\'alez D'Le\'on, Hanusa, and Yip~\cite{DHY}, and the strongly planar case of~\cite{PLAN}) are realized as well-ordered framed augmentations. We posit that several key properties of these settings may extend to the well-ordered case, e.g., some theory of framing posets.}

\subsection{Results about the well-ordered case}

We first define well-ordered framed augmentations.

\begin{defn}
	Let $G$ be a DAG with netflow vector $\a$ and let $(\hat G,\hat\a,\hat\FF)$ be a framed augmentation of $(G,\a)$. Let $L$ be a layering of $(\hat G,\hat\a,\hat\FF)$. Let $\Phi_L$ be the map from $V_X$ to $V_Y$ which sends a vertex $s_i\in V_X$ to $t(L_i)$, where $L_i$ is the route of $L$ beginning at $s_i$.
	We say that $(\hat G,\hat\a,\hat\FF)$ is \emph{well-ordered} if for any two layerings $L$ and $M$, we have $\Phi_L=\Phi_M$.
\end{defn}

It is immediate that any augmentation with a unique outflow edge is well-ordered, as every $\Phi_L$ will send every vertex of $V_X$ to the unique vertex of $V_Y$. The planar-framed augmentation on the right of Figure~\ref{fig:bigaug} has two outflow edges but is still well-ordered: its 9 layerings are drawn in Figure~\ref{fig:bigex} labelling the corresponding integer point of its flow polytope, and one may check that $\Phi_L(s_1)=t_1$ and $\Phi_L(s_2)=t_2$ for all layerings $L$.
On the other hand, every framed augmentation of the complete bipartite graph $(K_{3,3},\a)$ of Example~\ref{ex:k33} fails to be well-ordered because its layerings will correspond to perfect matchings of $K_{3,3}$, no two of which give the same map from sources to sinks (i.e., $\Phi_L\neq\Phi_M$ for any choice of distinct layerings $L$ and $M$).

In the well-ordered case, layering-simplices may be verified using a straightforward pairwise compatibility condition.

\begin{defn}
	Let $(\hat G,\hat\a,\hat\FF)$ be a well-ordered framed augmentation of $(G,\a)$. Let $L$ and $M$ be layerings of $(\hat G,\hat\a,\hat\FF)$. We say that $L$ and $M$ are \emph{noncrossing} if $\routes(\{L,M\})$ is a route-clique and without loss of generality, $L_i\leq_{\s_i}^+ M_i$ for all $i\in S_\a$.
	A \emph{layering-clique} is a set of pairwise noncrossing layerings.
\end{defn}

\begin{prop}\label{prop:ltr}
	A set of layerings of a well-ordered framed augmentation $(\hat G,\hat\a,\hat\FF)$ is a layering-clique if and only if it is a layering-simplex.
\end{prop}
\begin{proof}
	Let $\L$ be a set of layerings of a well-ordered famed augmentation $(\hat G,\hat\a,\hat\FF)$ of $(G,\a)$.
	Index $\L=\{L^1,\dots,L^m\}$ so that $L^1<_\src^+\dots<_\src^+L^m$. 

	First, suppose that $\L$ is a layering-clique.
	Then for any $j\in[S_\a]$ and $i_1<i_2\in[m-1]$ we have $L^{i_1}_j\leq_{\s_j}^+L^{i_2}_i$. We will show that this set satisfies the two conditions of Definition~\ref{defn:layering-simplexB}. Condition (1) is directly assumed.

	We now show (2). Take $i\in[m]$. Because $L^i,\dots,L^{i+1}$ are pairwise noncrossing, for every $j\in[S_\a]$ the route $L^i_j$ is the $<_{s_j}^+$-minimal route of $\routes(\{L^i,\dots,L^{i+1}\})$ beginning at $s_i$. It follows that
	\[L^i=\min_{<_\src^+}\left\{L\in\text{Layerings}(\hat G,\hat\a,\hat\FF)\ :\ \routes(\{L\})\subseteq\routes(\{L^i,\dots,L^m\})\right\}\]
	so we have shown (2).
	This completes the proof that $\L$ is a layering-simplex in this case.

	To prove the converse, we will assume that $\L$ is a layering-simplex and show that the layerings of $\L$ are pairwise noncrossing.
	To this end, take two elements $i<i'$ of $[m]$.
	We will show that $L^i$ and $L^{i'}$ are noncrossing by showing that $L^i_j\leq_{s_j}L^{i'}_{j}$ for every index $j\in[S_\a]$.
	Suppose to the contrary that there exists $j\in[S_\a]$ such that $L^i_j>_{s_j}L^{i'}_j$. Then the layering $L'$ obtained by replacing $L^i_j$ with $L^{i'}_j$ satisfies $\routes(\{L'\})\subseteq\routes(\{L^i,\dots,L^m\})$ and $L'<_\src^+ L$, contradicting Definition~\ref{defn:layering-simplexB}~(2). This completes the proof that any two layerings of $\L$ are noncrossing, hence that $\L$ is a layering-clique.
\end{proof}

It follows immediately from Proposition~\ref{prop:ltr} that the simplicial complex of layering-simplices in the well-ordered case is flag.

\begin{remk}\label{remk:JJ}
	Note that when $(\hat G,\hat\a,\hat\FF)$ is well-ordered, neither compatibility of routes nor the characterization of layering-simplices given in Proposition~\ref{prop:ltr} uses the order $<_{\FF,\src}$. Hence, in the well-ordered case this order never affects the framing triangulation. Outside of the well-ordered case, this order $<_{\FF,\src}$ will be necessary to define a triangulation (see, e.g., Example~\ref{ex:k33}, which is symmetric except for the order $<_{\FF,\src}$).
\end{remk}

We now give two examples of framed augmentations which are well-ordered.

\begin{example}\label{ex:augpolyg}
	Observe that the framed augmentation of Figure~\ref{fig:augpoly} (or Example~\ref{ex:augpoly}) is well-ordered, and hence the layering-simplex complex may be verified using the pairwise compatibility condition of Proposition~\ref{prop:ltr} rather than the layering-simplex definition of Definition~\ref{defn:layering-simplexB}.
\end{example}

\begin{example}\label{ex:2s1t}
	Let $(G,\a)$ be the DAG on the top-left of Figure~\ref{fig:2s1t} with netflow vector $(1,1,-2)$ as labelled in blue. A framed augmentation $(\hat G,\hat\a,\hat\FF)$ is shown below, with the inflow edges and vertices in green and orders of $\hat\FF$ induced from the planar embedding. The total order on $V_X$ is not impactful by Remark~\ref{remk:JJ}.
	The flow polytope $\F_G(\a)$ is integrally equivalent to a $2\times1$ rectangle, shown on the right with the vertices indexed by the corresponding layering. Vertices whose layerings are compatible are connected by a solid or dotted line, drawing the framing triangulation.

	For example, the layerings $L:=\{x_1\alpha_2\beta_2y_1,x_2\beta_1y_1\}$ and $M:=\{x_1\alpha_1\beta_1y_1,x_2\alpha_2y_1\}$ fail to be compatible because they are not noncrossing:
	\[
		L_{1}=x_1\alpha_2\beta_2y_1>_{s_1}^+x_2\alpha_1\beta_1y_1=M_1\text{ but }
		L_2=x_2\alpha_1y_1<_{s_2}^+x_2\alpha_2y_1=M_2.
	\]
	On the other hand, $L:=\{x_1\alpha_2\beta_2y_1,x_2\beta_2y_1\}$ and $M:=\{x_1\alpha_2\beta_1y_1,x_2\beta_1y_2\}$ are noncrossing with $L>_\src^+ M$, but they are not compatible because the routes $M_1$ and $L_2$ are incompatible.

	\begin{figure}
		\centering
		\def\svgscale{.38}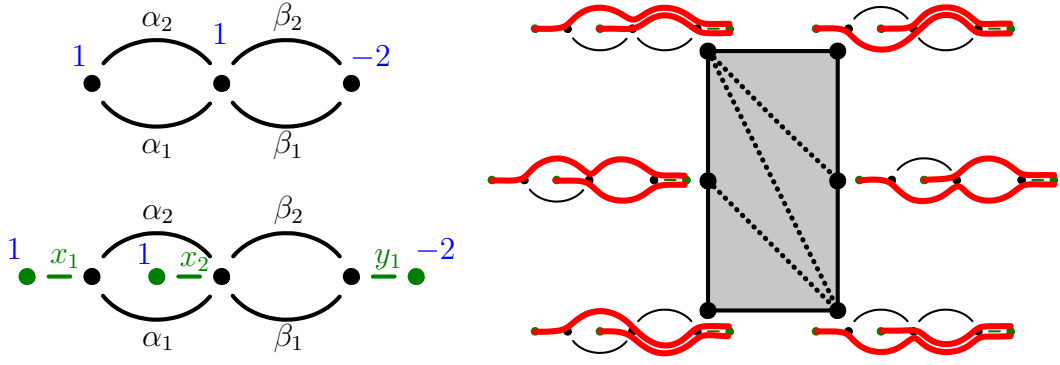
		\caption{A tuple $(G,\a)$, a well-ordered framed augmentation $(\hat G,\hat\a,\hat\FF)$, and its framing-triangulated flow polytope.}
		\label{fig:2s1t}
	\end{figure}
\end{example}

\begin{example}\label{ex:planar}
	Let $(G,\a)$ be as drawn on the left of Figure~\ref{fig:bigaug}, with $\a=(1,1,0,-1,-1)$. The right of the figure shows a framed augmentation $(\hat G,\hat\a,\hat\FF)$. Figure~\ref{fig:bigex} shows the three-dimensional flow polytope $\F_G(\a)$ on the bottom with its integer points labelled by layerings of $(\hat G,\hat\a,\hat\FF)$. The 8 maximal layering-simplices are drawn on the top of Figure~\ref{fig:bigex}, with the bottom one highlighted in teal along with its corresponding polyhedral simplex of the framing triangulation. The maximal layering-simplices are arranged as the dual graph of the framing triangulation -- i.e., with edges connecting neighboring maximal simplices in the framing triangulation.
	\begin{figure}
		\centering
		\def\svgscale{.38}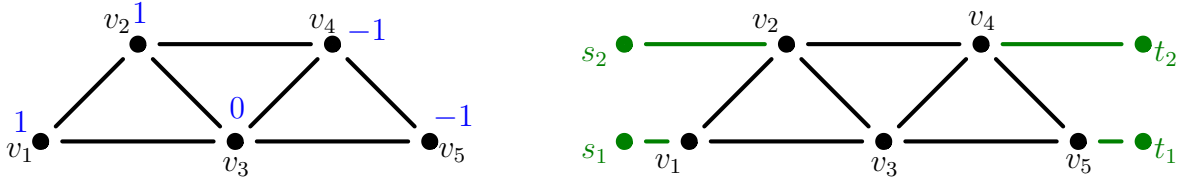
		\caption{A strongly planar embedding $(\G,\a)$ (left) with a strongly planar balanced augmentation (right).}
		\label{fig:bigaug}
	\end{figure}
	\begin{figure}
		\centering
		\def\svgscale{.30}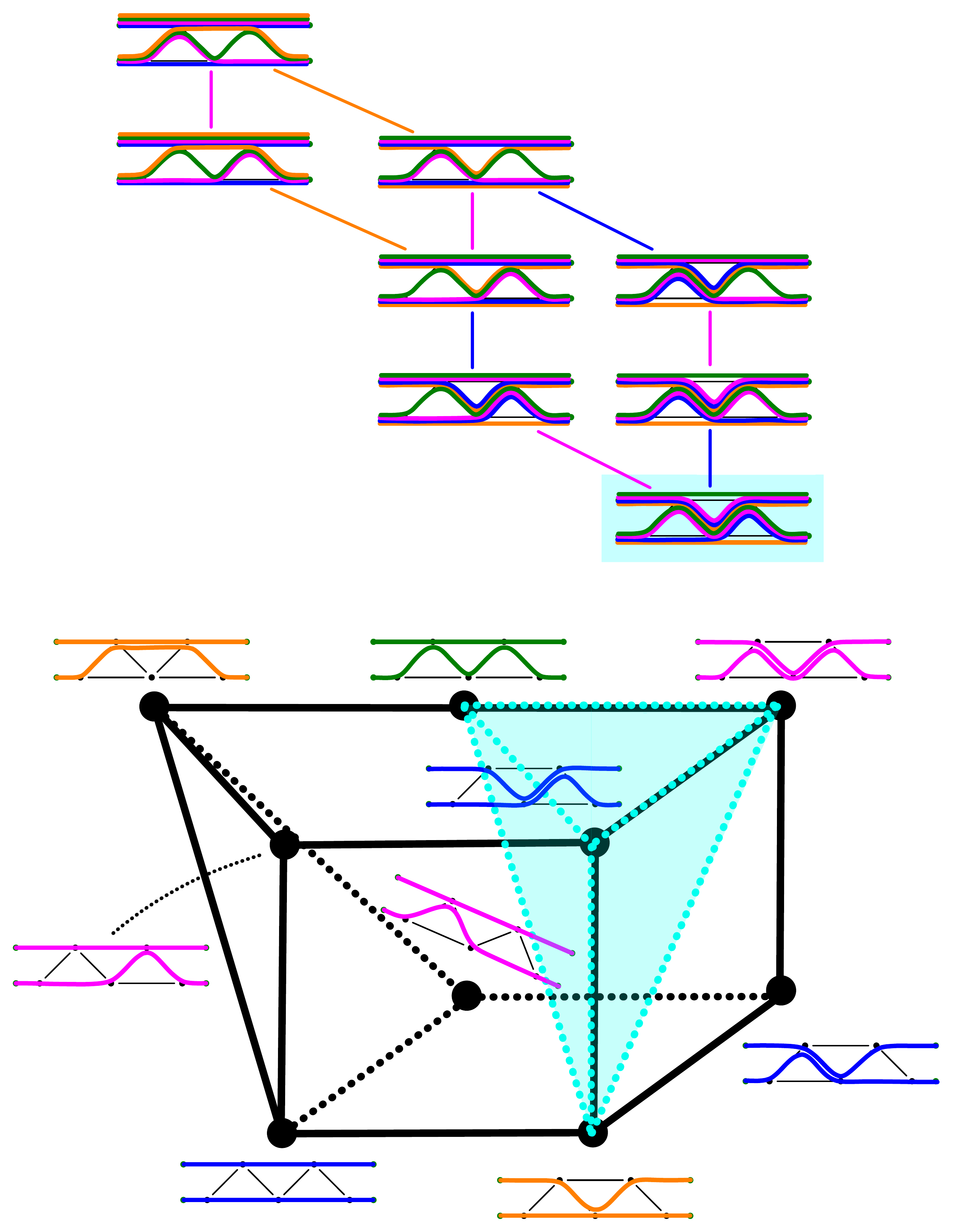
		\caption{The flow polytope $\F_G(\a)$ of the DAG of Figure~\ref{fig:bigaug} and the dual graph of its framing triangulation represented via maximal layering-simplices.}
		\label{fig:bigex}
	\end{figure}
\end{example}

\section{Relation to Existing Theories of Framing Triangulations}
\label{sec:special-cases}

We now quickly relate the framing triangulations introduced in this paper with the existing theories of framing triangulations in the literature defined for special classes of DAGs with netflow vectors. We begin by arguing that our theory matches that of Danilov, Karzanov, and Koshevoy~\cite{DKK} in the case when $G$ has one source and one sink and $\a$ is the unit netflow vector. {We then show more generally that it matches the theory of Gonz\'alez D'Le\'on, Hanusa, and Yip~\cite{DHY} when $G$ is arbitrary and $\a$ has a unique negative entry. Finally, we show that it matches the author's previous work~\cite{PLAN} when the framing $\FF$ is induced by a strongly planar embedding of $(G,\a)$.}

\subsection{DKK-framing triangulations}

If $G$ is a DAG with one source and one sink and $\a=(1,0,\dots,0,-1)$ is the unit netflow vector, then it is immediate that $\a$ is conservationist.
Moreover, a DKK-framing $\FF_{\dkk}$ gives rise to a framing $\FF_{\con}$ in the sense of Definition~\ref{defn:framing} by arbitrarily choosing a total order of the edges incident to the source and letting $<_{\FF,\src}$ be the trivial order on the set consisting of the unique source vertex.

It is immediate that two routes $p$ and $q$ of $G$ are compatible with respect to $\FF_\dkk$ if and only if they are compatible with respect to $\FF_\con$. Moreover, a layering $L$ of $(G,\a,\FF_\con)$ is merely a set consisting of a single route $L=\{p\}$, and $\J(L)=\I(p)$. 
Finally, for any set $\L$ of distinct layerings of $(G,\a)$ it is immediate that Definition~\ref{defn:layering-simplex}~(2) is satisfied, hence the only substantive condition for $\L$ to be a layering-simplex is that the underlying set of routes is pairwise compatible. This shows that the DKK-framing triangulation induced by $\FF_\dkk$ is equal to the framing triangulation of $(G,\a,\FF_\con)$, so our conservationist framing triangulations generalize those of Danilov, Karzanov, and Koshevoy~\cite{DKK}.

We remark again (see Remark~\ref{remk:DKK}) that the order of the framing $\FF_\con$ of the outgoing edges to the source vertex has no effect on compatibility of routes, hence does not affect the framing triangulation. The order $<_{\FF_\con,\src}$ is only relevant in Condition~(2) of Definition~\ref{defn:layering-simplex}, which becomes nontrivial when layerings have cardinality greater than one.

\subsection{{Framing triangulations when $\a$ has a unique negative element}}

As a consequence of a theory developed in a concurrent work~\cite{DHY}, Gonz\'alez D'Le\'on, Hanusa, and Yip give a theory of framing triangulations and framing posets in the setting where $G$ is arbitrary and the netflow vector $\a$ has a unique negative entry.
For example, the pair $(G,\a)$ on the left of Figure~\ref{fig:aug} has a unique vertex with negative netflow (see also the framing triangulation given in Figure~\ref{fig:augpoly}).

They define the \emph{augmented graph}~\cite[Definition 5.1]{DHY} of such a pair $(G,\a)$, which is the same as our augmentation $(\hat G,\hat\a)$ of $(G,\a)$ containing exactly one outflow edge from the vertex of $G$ with negative netflow (with the caveat that they consider inflow edges to be ``half-edges'' with no tail and outflow edges to be ``half-edges'' with no head).

They then define framings on augmented graphs.
The only difference between their notion of framings and ours is that we consider the order $<_{\FF,\src}$ to be a part of the framing $\FF$ which may be independently varied with the orders $<_{\FF,\hinn(i)}$, while in~\cite{DHY} they use the order on the vertex set $[n]$ and the orders $<_{\FF,\hinn(i)}$ to induce a total order on the inflow half-edges. This does not restrict the class of triangulations arising as framing triangulations in the well-ordered setting; see Remark~\ref{remk:JJ}.
Note that when $\a$ has a unique negative entry, the framed augmentation $(\hat G,\hat\a,\hat\FF)$ with a unique outflow vertex $t$ (i.e., the one corresponding to the augmented graph in the sense of~\cite{DHY}) is automatically well-ordered since $\Phi_L(s_i)=t$ for all $i\in S_\a$.

They define \emph{route-matchings}~\cite[Definition 5.3]{DHY}, which are the same as our layerings. Route-matchings form \emph{cliques}~\cite[Definition 5.6]{DHY}, which are the same as our layering-simplices using the characterization given in Proposition~\ref{prop:ltr}.
In this way, their framing triangulation result~\cite[Proposition 6.3]{DHY} is the same as our framing triangulation result~\ref{cor:triangulation} in the setting where $\a$ has a unique negative entry.

The work~\cite{DHY} is expansive and their framing triangulation result is connected with many other combinatorial objects; moreover, they give connections to the Lidskii volume formulas of Baldoni and Vergne. It would be interesting to see if some of this theory may generalize to the well-ordered setting.

\subsection{Framing triangulations in the strongly planar setting}

In~\cite{PLAN}, the author developed a theory of framing triangulations and framing posets for a notion of \emph{strongly planar} DAGs which may have multiple sources and multiple sinks.

\begin{defn}[{\cite[Definition 4.1]{PLAN}}]
	\label{defn:splanar}
	A planar embedding $\G$ of a nondegenerate DAG $G$ is \emph{strongly planar} if
	\begin{enumerate}
		\item for every directed edge $e:i\to j$ of $\G$, the x-coordinate of $i$ is strictly less than the x-coordinate of $j$,
		\item each edge $(i,j)$ is embedded into the plane as the graph of a piecewise differentiable function of $x$,
		\item every source and every sink of $\G$ is incident to the exterior face of the planar embedding $\G$, and
		\item every source vertex of $\G$ has the same x-coordinate, and every sink vertex of $\G$ has the same x-coordinate.
	\end{enumerate}
\end{defn}

For example, the embedding $\G$ of the DAG of Figure~\ref{fig:bigaug} is strongly planar.

\begin{remk}
	Definition~\ref{defn:splanar} was inspired by the strongly planar DAGs with one source and one sink whose planar-framed DKK triangulations were shown by M\'esz\'aros, Morales, and Striker~\cite{MMS} to be integrally equivalent to Stanley's triangulation of the order polytope of a corresponding poset.
\end{remk}

\begin{defn}
	Let $G$ be a DAG with strongly planar embedding $\G$ and netflow vector $\a$. 
	Suppose further that $\a$ is positive on sources, negative on sinks, and zero on internal vertices.
	We say that $(\G,\a)$ is \emph{strongly planar} if every vertex $i\in V$ with a nonzero netflow $a_i\neq0$ is incident to the exterior face of the planar embedding of $\G$.
\end{defn}

The tuple $(\G,\a)$ of Figure~\ref{fig:bigaug} is strongly planar.

Given a strongly planar $(\G,\a)$, choose the augmentation $(\hat G,\hat\a)$ of $(G,\a)$ such that each sink vertex $i\in V$ is incident to $|\a_i|$ outflow edges. Then $\hat G$ has $|E_X|=S_\a$ source vertices and $|E_Y|=S_\a$ sink vertices, making it \emph{balanced} in the sense of~\cite{PLAN}, meaning that it has the same number of sources and sinks. Moreover, one easily extends the strongly planar embedding $\G$ of $G$ to a strongly planar embedding $\hat\G$ of $\hat G$ (e.g., the right of Figure~\ref{fig:bigaug}).
Let $s_1,\dots,s_{S_\a}$ be the sources of $\hat G$ ordered bottom to top in the embedding $\hat\G$, and let $t_1,\dots,t_{S_\a}$ be the sinks ordered bottom to top.

Define a framing $\hat\FF_\G$ on the augmentation $(\hat G,\hat\a)$ by letting each $<_{\hat\FF_\G,\inn(i)}$ order the incoming edges to $i$ bottom to top with respect to the embedding $\hat\G$, and similarly letting each $<_{\hat\FF_\G,\hout(i)}$ order bottom to top. The total order $<_{\hat\FF_\G,\src}$ of the vertices $V_X$ will not affect the framing triangulation by Remark~\ref{remk:JJ}; regardless one may order the source vertices bottom to top with respect to the embedding $\G$.

It is then immediate that a layering of the strongly planar balanced DAG $(\hat\G,\hat\a)$ in the sense of~\cite{PLAN} is precisely a layering of the framed augmentation $(\hat G,\hat \a,\hat\FF_\G)$.
Then~\cite[Proposition 5.13]{PLAN} states that for any layering $L$, the route $p\in L$ beginning at $s_i$ must end at $t_i$ for every $i\in[S_\a]$, hence $(\hat G,\hat\a,\hat\FF_\G)$ is well-ordered.
The notion of layering-simplices defined in~\cite[Definition 5.16]{PLAN} is the same as the characterization of layering-simplices shown in Proposition~\ref{prop:ltr} in the well-balanced case, hence the framing triangulation result~\cite[Corollary 5.19]{PLAN} is a special case of Corollary~\ref{cor:triangulation}.
See Figure~\ref{fig:bigex} for an example.

The work~\cite{PLAN} goes on to give the structure of a \emph{framing poset} to the maximal simplices of the framing triangulation in the strongly planar case, which agrees with the framing lattices of von Bell and Ceballos~\cite{vBC} when $(G,\a)$ is strongly planar with one source and one sink and $\a$ is the unit netflow vector. It would be interesting to see if a theory of framing posets may be developed more generally for the well-ordered case.

\begin{conj}
	If $(\hat G,\hat\a,\hat\FF)$ is a well-ordered framed augmentation of a DAG $G$ with netflow vector $\a$, then the dual graph of its framing triangulation is the Hasse diagram of a poset.
\end{conj}

\bibliographystyle{alphaurl}
\bibliography{biblio} 

\end{document}